\documentclass[a4paper]{article}
\usepackage{amsfonts,amsmath,amssymb,amsthm,url}
\usepackage[all]{xypic}
\usepackage[noadjust]{cite}

\usepackage{epsfig}
\usepackage[british]{babel}
\usepackage{hyperref}
\usepackage{multirow}
\usepackage{lscape}
\usepackage{longtable}
\usepackage{pdflscape}
\usepackage{caption}
\usepackage[a4paper]{geometry}
\usepackage{enumitem}

\numberwithin{equation}{section}

\newtheorem{theorem}{Theorem}[section]

\theoremstyle{definition}

\newtheorem{proposition}[theorem]{Proposition}
\newtheorem{lemma}[theorem]{Lemma}
\newtheorem{corollary}[theorem]{Corollary}
\newtheorem{remark}[theorem]{Remark}
\newtheorem{definition}[theorem]{Definition}
\newtheorem{algorithm}[theorem]{Algorithm}
\newtheorem{algorithmsketch}[theorem]{Algorithm Sketch}

\newcommand{\HH}{\mathcal{H}}

\newcommand{\CC}{\mathbf{C}}
\newcommand{\ZZ}{\mathbf{Z}}
\newcommand{\RR}{\mathbf{R}}
\newcommand{\QQ}{\mathbf{Q}}
\newcommand{\FF}{\mathbf{F}}
\newcommand{\QQbar}{\overline{\QQ}}

\newcommand{\Gal}{\mathrm{Gal}}
\newcommand{\GL}{\mathrm{GL}}

\newcommand{\End}{\mathrm{End}}

\newcommand{\SL}{\mathrm{SL}}
\renewcommand{\mod}{\ \mathrm{mod}\ }
\newcommand{\PP}{\mathbf{P}}

\newcommand{\Aut}{\mathrm{Aut}}
\newcommand{\kbar}{\overline{k}}

\newcommand{\tr}{\mathrm{tr}}
\newcommand{\rightaction}{\mathrel{\rotatebox{90}{$\circlearrowright$}}}
\newcommand{\fnew}{F^\dagger}
\newcommand{\Den}{\mathfrak{D}}

\title{Examples of CM curves of genus two defined over the reflex field}

\author{Florian Bouyer\thanks{University of Warwick,
\url{http://www.warwick.ac.uk/fbouyer}
email: \url{F.Bouyer@Warwick.ac.uk},
supported by the University of Warwick
Undergraduate Research Scholarship Scheme (URSS)}  
\ and Marco Streng\thanks{Universiteit Leiden,
\url{http://www.math.leidenuniv.nl/~streng},
email: \url{Marco.Streng@gmail.com},
partially supported by EPSRC grant number
EP/G004870/1 and by NWO Veni project number 639.031.243}
}

\begin{document}
 
\maketitle

 \begin{abstract}
\noindent 
Van Wamelen~\cite{vanwamelen} lists
19 curves of genus two over $\QQ$
with complex multiplication (CM).
However, for each
curve,
the CM-field turns
out to be cyclic Galois over~$\QQ$,
and the 
generic 
case of
a non-Galois
quartic CM-field
did not feature in this list.
The reason is that the field
of definition in that case always contains
the real quadratic subfield of the reflex field.

We extend Van Wamelen's list to include curves
of genus two defined over this real quadratic field.
Our list therefore contains
the smallest ``generic'' examples
of CM curves of genus two.

We explain our methods for obtaining this list, including a
new height-reduction
algorithm for arbitrary hyperelliptic curves
over totally real number fields.
Unlike Van Wamelen, we also give a proof of our list,
which is made possible by our implementation
of denominator bounds of Lauter and Viray for
Igusa class polynomials.\\
\\
cite as: 
LMS Journal of Computation and Mathematics, volume 18 (2015), issue 01, pp.~507--538
\end{abstract}

\section{Introduction}

We say that a curve $C/k$ of genus $g$ has \emph{complex multiplication}
(CM) if the endomorphism ring of its Jacobian over $\overline{k}$
contains
an order in a number field~$K$ of degree~$2g$.
Curves of genus one
(elliptic curves)
and two with complex multiplication
are important in the
\emph{CM-method} for constructing (hyper)elliptic curves
for cryptography, and for construction of
\emph{class fields} from class field theory.

It is well known that there exist exactly 13 elliptic curves
over $\QQ$ with complex multiplication
(see e.g.~\cite[Theorem~7.30(ii)]{cox}).
Analogously, Van Wamelen~\cite{vanwamelen} gives a list of 19
curves of genus two over $\QQ$
with CM
by a maximal order
(proven in~\cite{correctnesswamelen, bisson-streng}).

In the genus-two case, the (quartic) CM-field~$K$
is either cyclic Galois,
biquadratic Galois, or non-Galois with Galois group~$D_4$.
Like Van Wamelen, we disregard the degenerate biquadratic case, as the corresponding
Jacobians are isogenous to a product of
CM elliptic curves.
Murabayashi and Umegaki~\cite{completenesswamelen} then show
that Van Wamelen's list is complete.
However, the list only contains examples of the cyclic case,
not the~$D_4$ case, because curves in the latter
case cannot be defined over~$\QQ$.

In this paper, we give a list of the simplest examples
of the $D_4$ case, namely those defined over 
certain real
quadratic extensions of~$\QQ$.
Our end result is as follows.
\begin{theorem}\label{thm:main}
For every row of the tables 1a, 1b, and 2b
on pages~\pageref{table1astart}--\pageref{table2bend},
let $K=\QQ[X]/(X^4+AX^2+B)$,
where $[D,A,B]$ is as in the first column of the table.
Then the curves $C:y^2=f(x)$
where $f$ is as in the last column
are exactly all curves
with complex multiplication
by the maximal order of~$K$,
up to isomorphism over $\QQbar$
and up to automorphism of $\QQbar$.

The number $a$ that may appear in the coefficients of
$f$ is as follows. In table 1b, let $D'=D$,
and in table 2b, let $[D',A',B']$ be as in the second
column. Let $\epsilon\in\{0,1\}$ be $D'$ modulo 4.
Then $a$ is a root of
$x^2+\epsilon x + (\epsilon-D')/4=0$.
\end{theorem}
 Section~\ref{sec:list} contains more detailed statements,
including an explanation of the
other columns.

P{\i}nar K{\i}l{\i}{\c{c}}er and the second-named author
are currently working on a proof of completeness.
That is, we believe that the first columns of
Tables 1a, 1b, and 2b contain exactly
the quartic
fields $K$ for which there exists a curve $C$
of genus two with $\mathrm{End}(J(C))=\mathcal{O}_K$
such that $C$ is defined over the
real quadratic subfield of the reflex field.

Now, let us give a more detailed overview of
our methods, which form the bulk of this paper.

Given a quartic CM-field~$K$, we compute the curves $C/\QQbar$
of genus two with $\End(J(C))\cong \mathcal{O}_K$
in three stages.
First we compute the Igusa invariants $i_1(C)$,
$i_2(C)$ and $i_3(C)$ as elements of~$\QQbar$.
Second we compute an arbitrary model of $C$
from its Igusa invariants.
Third we reduce this model to a small model, that is,
a model with integer coefficients of only
a handful of digits.

Section~\ref{sec:invcmmain} explains Igusa
invariants and how to compute them,
that is, the first stage of our algorithm.
Our new contribution there is an implementation
of denominator bounds of Lauter and
Viray~\cite{lauter-viray-denominators},
which allows us to be the first to systematically
compute and prove correctness
of CM Igusa invariants.

Section~\ref{sec:mestre}
quickly reviews Mestre's algorithm
for computing a model of $C$ from its
invariants; the middle stage of our algorithm.

Section~\ref{sec:reduction}
explains how to go from any model
to a small model
which is the final stage of our algorithm.
Mestre's algorithm constructs curves with given invariants,
but these curves have coefficients of thousands of digits,
so we use a reduction algorithm to reduce
the coefficient size.
Our main new contribution there is
a reduction algorithm based on 
Stoll and Cremona~\cite{stoll-cremona},
including an implementation.

We applied this algorithm to fields
in the Echidna database~\cite{echidna}
and obtained our tables.
Section~\ref{sec:list} gives a detailed version
of Theorem~\ref{thm:main},
explaining all columns of the tables.
We end with a cryptographic application in Section~\ref{sec:application}.

\subsection*{Acknowledgements}
The authors would like to thank
Bill Hart for his help with
the factoring software GMP-ECM and CADO-NFS,
Jeroen Sijsling for useful discussions
about models and invariants,
Damiano Testa as advisor of the first-named author,
Kristin Lauter and Bianca Viray for help with their
formulas~\cite{lauter-viray-denominators},
Christophe Ritzenthaler for the reference~\cite{milne-definition},
and the anonymous referee for helpful suggestions
for the improvement of the exposition.

\section{Invariants and complex multiplication}
\label{sec:invcmmain}

\subsection{Overview}

The first stage of our algorithm is,
given a quartic CM-field~$K$,
to obtain the Igusa invariants of
the curves $C$ of genus two
with $\End(J(C))\cong \mathcal{O}_K$.
There are various practical methods for
doing so (complex analytic, $p$-adic, or using
the Chinese Remainder Theorem)
numerically up to some precision,
and results have been collected also
in the Echidna database~\cite{echidna}.

However, we want a proven output
and the only prior proven method
is that of Streng~\cite{runtime},
which in practice is too slow
even for the relatively small discriminants
we consider.
This section explains
the complex analytic method and
shows how to make it into a method
that is both practical and proven,
using denominator bounds
of Lauter and Viray~\cite{lauter-viray-denominators}.

Our algorithm for this part closely follows
that of~\cite{runtime}, so
many of the details and proofs of what follows
can be found there.

\subsection{Igusa invariants and Igusa class polynomials}
\label{ssec:icp}

Since the goal of this section is to compute
Igusa invariants, let us begin by reviewing
them and explaining how they are represented
by Igusa class polynomials.
For details on Igusa invariants, see
Igusa~\cite{igusa},
and for details on Igusa class polynomials,
see Streng~\cite{runtime}.

For an elliptic curve~$E/k$, the $j$-invariant $j(E)\in k$
uniquely specifies
the isomorphism class of~$E$ over~$\kbar$.

For a (smooth, projective, geometrically
irreducible algebraic) curve $C/k$
of genus two, the situation is a bit more complicated.
For simplicity, we assume $k$ has characteristic
different from $2$, $3$, $5$.
Every curve of genus two is hyperelliptic, 
that is, is birational to an affine curve
$y^2 = f(x)$ where $f\in k[x]$ has degree $5$ or~$6$
and no roots of multiplicity~$>1$.

The \emph{Igusa-Clebsch invariants}
$I_2$, $I_4$, $I_6$, and $I_{10}$
are polynomials in the coefficients of~$f$.
They can be found in Igusa~\cite{igusa},
where they are denoted $A$, $B$, $C$, $D$
and are based on invariants of Clebsch.
They are also available in the software packages
Magma~\cite{magma} and Sage~\cite{Sage}.
The last invariant, $I_{10}$, is
related to the discriminant of~$f$
and
is always non-zero.
Actually, for efficiency we use
$I_6' = \frac{1}{2} (I_2I_4-3I_6)$
instead of $I_6$,
but one can easily go back and forth using
$I_6 = \frac{1}{3} (I_2I_4-2I_6')$.
See~\cite{runtime}.

Isomorphic hyperelliptic curves have Igusa-Clebsch
invariants that are equal up to a weighted scaling.
In fact, for curves $C$ and $C'$, we have $C_{\kbar} \cong C_{\kbar}'$
if and only if
there is a $\lambda\in\kbar^*$
such that for $j=2,4,6,10$ we have
$I_j(C)=\lambda^j I_j(C')$.

In more geometric language,
if $\PP^{2,4,6,10}$ is the
\emph{weighted projective space} of weights $(2,4,6,10)$
with the Igusa-Clebsch invariants as coordinates,
then the subspace $\mathcal{M}_2\subset \PP^{2,4,6,10}$ 
defined by $x_{10}\not=0$
is a coarse moduli space of genus-two curves
in characteristic not dividing $2\cdot 3\cdot 5$.

Following~\cite{runtime},
we make a choice of three \emph{absolute Igusa invariants}
$i_1, i_2, i_3$,
which generate the function field of the
moduli space:
\[ i_1 = \frac{I_4^{\phantom{2}}I_6'}{I_{10}^{\phantom{2}}},\quad
i_2 = \frac{I_2^{\phantom{2}}I_4^2}{I_{10}^{\phantom{2}}},\quad
i_3 = \frac{I_4^5}{I_{10}^2}.\]
Given a curve $C$, if $I_2(C)\not=0$, then $C$
is uniquely specified by $i_1,i_2,i_3$ because
of
$I_4 = I_2^{\vphantom{-}2}i_2^{-2} i_3^{\phantom{2}}$,
$I_6' = I_2^{\vphantom{-}3} i_1^{\vphantom{-2}} i_2^{-3} i_3^{\vphantom{-2}}  $,
and $I_{10} = I_2^{\vphantom{-}5} i_2^{-5} i_3^{\vphantom{-}2}$.
And if we are unlucky enough to find $I_2(C)=0$,
then variants of these absolute Igusa invariants
will do the trick~\cite{cardona-quer}.

The Igusa class polynomials of a quartic CM-field $K$
are polynomials that specify the values of
$i_n(C)$ where $C$ has CM by $\mathcal{O}_K$.
In detail, they are
$$ H_{K,1} = \prod_{C} (X-i_1(C)),\quad
\widehat{H}_{K,n} = \sum_{C} i_n(C)\prod_{C'\not\cong C}
(X-i_1(C'))\in\QQ[X],
$$
for $n\in\{2,3\}$,
where $C$ and $C'$ range over isomorphism classes
of curves with $\End(J(C))\cong \mathcal{O}_K$.

One can recover the Igusa invariants
$i_1(C), i_2(C), i_3(C)$ from these polynomials
by taking all roots $i_1(C)$ of $H_{K,1}$
and letting, for $n\in\{2,3\}$,
$$i_n(C) = \widehat{H}_{K,n}(i_1(C))/H_{K,1}'(i_1(C)),$$
assuming $H_{K,1}$ has no roots of multiplicity $>1$.
Again, if we are unlucky, there are ways to work around
this (\cite[Section III.5]{phdthesis}).

In particular, our goal in Section~\ref{sec:invcmmain}
is to compute $H_{K,1}$ and $\widehat{H}_{K,n}$.

\subsection{Complex approximation of Igusa class polynomials}
\label{sec:CCapprox}

Streng~\cite{runtime}
explains in detail how to compute complex numerical
approximations of $H_{K,1}$ and~$H_{K,n}$.
The only way in which we deviate from the method of~\cite{runtime}
is by using interval arithmetic in our implementation.

Interval arithmetic is a computational model for~$\RR$
where real numbers are represented by intervals
that contain them. For intervals $a_i$
and a map $f:\RR^n\rightarrow \RR$, when asking
the computer for $f(a_1,\ldots, a_n)$, it
returns an interval
that contains $f(x_1,\ldots,x_n)$ for all $x_i\in a_n$.
Rounding is always done in such a way that the intervals
are guaranteed to be correct, hence the user does not
have to estimate rounding errors by hand.

Given any integer $N>0$,
we compute $F_1$, $F_2$, $F_3\in\QQ[X]$ such that
the polynomials
$F_1-H_{K,1}$ and $F_n-\widehat{H}_{K,n}$ 
are proven to have coefficients of absolute value
$<2^{-N}$ as follows.
Fix some precision $>N$ and do (floating point)
interval arithmetic to that precision.
If the output intervals are not small enough,
then double the precision and start over.

\subsection{Denominators}

\subsubsection{Using denominator bounds}

The first stage of our algorithm for computing
CM curves is computing their Igusa invariants,
and we have so far determined that it suffices to
compute the Igusa class polynomials
$H_{K,1}$, $\widehat{H}_{K,2}$, $\widehat{H}_{K,3}\in \QQ[X]$.
In this section, we give references
for how to compute a positive
integer $\Den{}=\Den{}_K$ such that 
$\Den{}H_{K,1}$, $\Den{}\widehat{H}_{K,2}$, and $\Den{}\widehat{H}_{K,3}$
all have \emph{integer} coefficients.
In particular, if we use the method
of Section~\ref{sec:CCapprox} to compute
approximations $F_1$, $F_2$, $F_3$ of
the Igusa class polynomials such that
all coefficients of $F_1-H_{K,1}$ and $F_n-\widehat{H}_{K,n}$ 
are proven to be of absolute value $<\frac{1}{2}\Den{}$,
then by rounding the coefficients
of $\Den{}F_i$ to the nearest integer
and dividing by $\Den{}$, we recover the Igusa class polynomials.

The first upper bounds on the primes dividing
the denominator of $H_{K, 1}$ and $\widehat{H}_{K,n}$
were given by Goren and Lauter~\cite{goren-lauter}.
More recently they~\cite{gorenlautertoappear}
also gave upper bounds
on the exponents with which these primes occur,
and combining these results leads to
a correct number $\Den{}$ as above.
This number is studied and used in~\cite{runtime},
but is too large to yield a practical algorithm.
An alternative of Bruinier and Yang~\cite{bruinier-yang,yang}
does give a very sharp number~$\Den$, but puts
too many restrictions on the quartic CM-field~$K$.
Fortunately, Lauter and Viray~\cite{lauter-viray-denominators}
managed to extend the latter
bounds to general number fields in a way that stays
sharp enough for our applications.

We have implemented the bounds of Lauter and
Viray~\cite{lauter-viray-denominators}
in Sage, and made the implementation
available at~\cite{cmcode}.
This finishes the first stage of our algorithm:
computing CM Igusa invariants.
We applied the denominator formulas
of~\cite{lauter-viray-denominators}
quite straightforwardly, but those
familiar with the formulas may wish to see a few more details.
We give these details in Section~\ref{ssec:detailslv},
but in order not to have to repeat the (complicated) formulas,
this may be of use only for those
who have~\cite{lauter-viray-denominators} close by.
Other readers may wish to skip to
Section~\ref{sec:mestre}.

\begin{remark}
\label{rem:padicandcrt}
Computing proven Igusa class polynomials is not
only possible with the complex analytic method,
but also with the methods based on $p$-adic numbers
(e.g.~\cite{ghkrw-2adic,ckl-3adic})
and the Chinese Remainder
Theorem
(e.g.~\cite{eisentrager-lauter}).
These methods first compute the
coefficients of the polynomials as elements of $\ZZ/N\ZZ$,
where $N$ is a large power of a small prime
or the product of a large set of small primes,
and then recognise the coefficients as elements of~$\QQ$.
For this final step to have a unique solution,
one needs to know an upper bound $b$
on the absolute value of the coefficient
(given by a crude low-precision complex analytic computation).
Suppose $N > 2b\Den{}$ is coprime to~$\Den{}$,
and suppose that a coefficient $c$ is
computed modulo~$N$, so we know $a=(c\mod N)\in\ZZ/N\ZZ$.
Then take the unique representative $r\in \ZZ$ 
of $a\Den{}\in\ZZ/N\ZZ$ with 
$|r|\leq b\Den{}$. The coefficient is $c=r/\Den{}\in\QQ$.
\end{remark}

\subsubsection{Implementation details}
\label{ssec:detailslv}

All fields $K$ in our tables, except for
the field $[257, 23, 68]$, satisfy
$\mathcal{O}_K=\mathcal{O}_{K_0}[\eta]$ for some $\eta\in K$.
For those fields, we use the bound
of \cite[Theorem 2.1]{lauter-viray-denominators}.
See \cite[Proof of Theorem~9.1]{yang}
for how exactly this applies to
Igusa class polynomials, where only the
constant coefficient $H_{K,1}$
is mentioned, though the proof
applies to all coefficients
of $H_{K,1}$ and $\widehat{H}_{K,n}$.

We used the obvious and straightforward way to
evaluate all the numbers occurring on the right
hand side of \cite[Theorem 2.1]{lauter-viray-denominators},
except for~$\mathcal{J}=\mathcal{J}(d_uf_u^{-2}, d_x, t)$.
For the number~$\mathcal{J}$, which counts solutions
to a ring embedding problem, we used~$0$
whenever \cite[Theorem 2.4]{lauter-viray-denominators} proves
it is~$0$, and we used the upper bound of
\cite[numbered displayed formula
in Theorem 2.4]{lauter-viray-denominators} otherwise.
These bounds turned out to be small enough
so that it took only a few hours to compute
all class polynomials.

For the field $K=[257, 23, 68]$, we chose ten different
$\eta\in \mathcal{O}_K$ such that
$I_\eta = [\mathcal{O}_K:\mathcal{O}_{K_0}[\eta]]$ is coprime
to all primes $p\leq D/4$,
where $D$ is the discriminant of~$K_0$. For each $\eta$
and each $\ell\nmid I_\eta$, we computed
the bound of \cite[Theorem 2.3]{lauter-viray-denominators}
on the $\ell$-valuation of the denominator
(and took $\infty$ as upper bound at $\ell\mid I_\eta$).
Then for each $\ell$, we took the minimum over all~$\eta$
of this valuation bound.
Finally, we sharpened the valuation bounds
further using
Goren and Lauter~\cite{gorenlautertoappear}.
This final bound took a little over half an hour
to compute, but was then small enough for our class
polynomial computation to finish within half an hour.
Indeed, the index $I_{\eta}$ had to be $> \Den/4$, which
made the bounds of \cite{lauter-viray-denominators}
hard to compute and far from sharp in this case.
We were advised afterwards by Kristin Lauter
that we did not have to exclude all primes $\leq \Den/4$,
and that \cite[Theorem 2.3]{lauter-viray-denominators}
also holds if one only avoids the primes dividing
the numbers~$\delta$ in their formulas.

It would be useful to have a fast algorithm for
computing~$\mathcal{J}$, rather than only bounds.
Fortunately, for our purposes, the bounds were
good enough.

\section{Mestre's algorithm}
\label{sec:mestre}

At this point, we have a number field~$k$ and
Igusa invariants in this number field,
and we wish to decide whether there is a
curve of genus two over $k$ with those Igusa invariants,
and if so, compute any model of the form $y^2=f(x)$
of that curve with $f\in k[x]$.
This is done by Mestre's algorithm,
which we will explain in this section.
Nothing in this section is new, and our reference
for this section is Mestre~\cite{mestre}.
Note that we do not care about the size of the
coefficients of $f$ yet, as long as we can compute it.
Reducing its size is Section~\ref{sec:reduction}.

Let $k$ be any field of characteristic not $2$,
$3$, or $5$.
Let $$\mathcal{M}_2(\kbar) = \{(x_2,x_4,x_6,x_{10})\in\kbar^4 
\mid x_{10}\not=0\} / k^*,$$
where $\lambda\in k^*$ acts by a 
weighted scaling
$\lambda(x_2,x_4,x_6,x_{10}) = (\lambda^2x_2,\lambda^4x_4,
\lambda^6x_6,\lambda^{10}x_{10})$.
We say that a point $x\in\mathcal{M}_2(\kbar)$
is defined over $k$ if 
$x\in \mathcal{M}_2(\overline{k})$
is stable under the action of $\Gal(\overline{k}/k)$.
One can show (using Hilbert's Theorem 90)
that this condition is satisfied if and only
if $x$ is the equivalence class
of a quadruple with for all $n\in k$, $x_n\in k$.
The \emph{field of moduli} $k_0$ of $C/\kbar$ is
smallest field over which
the point $x=(I_n(C))_n\in\mathcal{M}_2(\kbar)$ is defined.
We say that a field $l\subset\overline{k}$ is a \emph{field of definition}
for $C$ if there exists a curve $D/l$ with $D_{\overline{k}}\cong C$.

Unlike the elliptic case,
there is no simple formula for $C$ given~$(I_n(C))_n$,
and~$C$ cannot always
be defined over its field of moduli.
There does exist an algorithm, due to Mestre~\cite{mestre},
that finds a model for~$C$ given~$x$,
but it involves solving a conic, which
is not always possible without extending the field.
When it is possible to solve a conic over the base field,
then it usually introduces large numbers, so
that the output polynomial may have coefficients that
are much too large to be practical.

In more detail, Mestre's algorithm works as follows.
First of all, assume that the curve $C$ with $x=(I_n(C))_n$
does not have any automorphisms other than the hyperelliptic
involution $\iota:(x,y)\mapsto (x,-y)$.
(If it does, then use the construction of Cardona and
Quer~\cite{cardona-quer} instead
of Mestre's.)
From the coordinates $x_n$ in the field of moduli~$k_0$, one constructs
homogeneous ternary forms $Q=Q_x$ and $T=T_x\in k_0[U,V,W]$
of degrees $2$ and~$3$
(for equations, see~\cite{mestre} or~\cite{our-trac-ticket}).
Let $M_x\subset \PP^2$ be the conic defined by $Q$.
If $M_x$ has a point over a field~$k\supset k_0$,
then this gives rise to a parametrisation
$\varphi : \PP^1\rightarrow M_x$ over~$k$.
Let $\varphi^*:k[U,V,W]\rightarrow k[X,Z]$ be the
ring homomorphism inducing this parametrisation.
We get a hyperelliptic curve $C_{\varphi} : Y^2 = \varphi^*(T)$,
i.e., $C_{\varphi}:y^2 = T(\varphi(x:1))$.
The curve $C_{\varphi}$
is a double cover of $\PP^1$, ramified at the six points of~$\PP^1$
that map (under $\varphi$) to the six zeroes of $T_x$ on~$M_x$.
\begin{theorem}[{Mestre~\cite{mestre}}]
  Given $x\in\mathcal{M}_2(k)$, assume that the curve
  $C/\kbar$ with $x=(I_n(C))_n$ 
  satisfies $\Aut(C)=\{1,\iota\}$.
\begin{enumerate}
 \item If $M_x(k)=\emptyset$, then $C$ has no model over $k$.
\item If $M_x(k)\not=\emptyset$, then $C_{\varphi}/k$ as above is a model of~$C$.
\end{enumerate}
\end{theorem}

We use Magma~\cite{magma} to solve conics over number fields
and we contributed our Sage implementation of Mestre's algorithm
to Sage~\cite{Sage}, where it is available (as of version 5.13)
through the command \verb!HyperellipticCurve_from_invariants!.

There are by the way many quadratic extensions $l\supset k$
over which it \emph{is} possible to solve the conic:
simply choose all but one of the coordinates for the conic point
at will
and solve for the remaining coordinate, which yields a conic
point over a quadratic extension $l\supset k$.

\section{Reduction}
\label{sec:reduction}

In the previous section we described Mestre's algorithm
for finding models of genus-two curves over
number fields~$k$.
However, these hyperelliptic models in practice have coefficients of hundreds of digits.
In this section we 
describe how the make hyperelliptic curve equations over~$k$ smaller.
We start by explaining the relation between twists of hyperelliptic curves
and an action of $\GL_2(k)\times k^*$ on binary forms.
The rest of the section then is about
$(\GL_2(k)\times k^*)$-reduction of binary forms,
and our algorithm consists of two parts:
\begin{enumerate}
\item Making a binary form integral with discriminant of small norm
(Section~\ref{ssec:discriminantreduction}).
\item Making the heights of the coefficients small
by $(\GL_2(\mathcal{O}_k)\times \mathcal{O}_k^*)$-transformations,
which preserve integrality and
affect the discriminant only by units (Section~\ref{ssec:stollcremona}).
\end{enumerate}

We give the reduction algorithm for binary forms
of general degree~$n$, though it only
applies to hyperelliptic curves in the case that~$n$ is even
and~$\geq 6$.

\subsection{Isomorphisms and twists}
\label{ssec:isomandtwist}

Fix an integer $n\geq 3$ and a field~$k$,
and let
$H_n(k)$
be the set of separable binary forms
of degree $n$ in $k[X,Z]$.
We interpret $F(X,Z)\in H_n(k)$ also as the pair
$(n,f(x))$, where $f(x)=F(x,1)\in k[x]$
is a polynomial of degree $n$ or $n-1$.
In the case where $n$ is even and $\geq 6$,
let $g=(n-2)/2$ and interpret~$F$
as the hyperelliptic curve $C=C_f=C_F$
of genus~$g$
given by the affine equation $y^2 = f(x)$.
We can also write~$C$ as the smooth curve given by
$Y^2 = F(X,Z)$ in weighted 
projective space $\PP^{(1,g+1,1)}$.

Given any element of~$H_{2g+2}(k)$,
we would like to find an isomorphic
hyperelliptic curve with coefficients
of small height,
so first we determine when two hyperelliptic
curves are isomorphic.
\newcommand{\matabcd}{\left(\begin{array}{cc} a & b\\ c & d\end{array}\right)}

Note the natural right group actions
of scaling and substitution for any~$n$,
\begin{align*}
H_n(k)&\rightaction k^* & :&\quad (F(X,Z),u)\mapsto uF(X,Z),\quad\mbox{and}\\
H_n(k)&\rightaction \GL_2(k)& :&\quad 
 (F(X,Z),A)\mapsto F(A\cdot(X,Z)),
\end{align*}
which together induce an action of $\GL_2(k)\times k^*$
on~$H_n(k)$.

In terms of the polynomial $f(x) = F(x,1)\in k[x]$,
the action is
$$f(x) \cdot \left[ \left(\begin{array}{cc} a & b\\ c & d\end{array}\right), u\right]
= u\ (cx+d)^{n} f\left(\frac{ax+b}{cx+d}\right).$$

Note that a hyperelliptic curve $C$ always has the identity
automorphism and the
\emph{hyperelliptic involution}
$\iota : C\rightarrow C: (x,y)\mapsto (x,-y)$.
We will often assume that these are the only
automorphisms.
\begin{proposition}\label{prop:twist}
Given any two $F$, ${\fnew}\in H_{2g+2}(k)$, 
assume $\mathrm{Aut}((C_F)_{\kbar})=\{1,\iota\}$.
Then $C_F$ and~$C_{\fnew}$ are isomorphic over~$\kbar$
if and only if $F$ and~$\fnew$
are in the same orbit under
$\GL_2(k)\times k^*$.
\end{proposition}
\begin{proof}
It is a standard result (see e.g.~\cite[p.~1]{cassels-flynn}
for the case of genus two) that two hyperelliptic 
curves $C_F$ and $C_{\fnew}$ in $H_n(k)$
are isomorphic over~$k$
if and only if they are in the same
orbit under $\GL_2(k)\times (k^*)^2$.
Using $\mathrm{Aut}(C_{\kbar})=\{1, \iota\}$,
  we get (see e.g.~\cite[Example~C.5.1]{hindry-silverman})
  that all twists, up to isomorphisms over $k$,
  are given by the action
  of $H^1(k, \{1,\iota\}) = k^*/k^{*2} = \{1\}\times (k^*/k^{*2})$.
\end{proof}

\begin{remark}\label{rem:isom}
If $\fnew=F\cdot [({a \atop c}{b\atop d}), v^2]$,
then an isomorphism $C_{\fnew}\rightarrow C_{F}$
is given by
$(x,y)\rightarrow (\frac{ax+b}{cx+d}, v^{-1}(cx+d)^{-g-1} y)$.
\end{remark}

By Proposition~\ref{prop:twist}, finding small-height
models over $k$
of hyperelliptic curves $C/k$
with $\mathrm{Aut}(C_{\kbar})=\{1,\iota\}$
is equivalent to finding small elements of $\GL_2(k)\times k^*$-orbits
of binary forms of even degree~$\geq 6$.
Lemma~\ref{lem:aut} in Section~\ref{ssec:theory}
will show that the hypothesis $\mathrm{Aut}(C_{\kbar})=\{1,\iota\}$
is satisfied for the curves we deal with,
except for one curve for which we do not need a reduction
algorithm.
If $\mathrm{Aut}(C_{\kbar})\not=\{1,\iota\}$,
then $\GL_2(k)\times k^*$-actions may be too restrictive, but
by Remark~\ref{rem:isom}, they do always give valid twists.

Our goal for the remainder of Section~\ref{sec:reduction}
is, given a binary form $F\in H_n(k)$, to find
a $\GL_2(k)\times k^*$-equivalent form
with small coefficients.
We start with computing a discriminant-minimal form in
Section~\ref{ssec:discriminantreduction},
followed by discriminant-preserving $\GL_2(\mathcal{O}_k)\times\mathcal{O}_k^*$-reduction
in Section~\ref{ssec:stollcremona}.

\subsection{Reduction of the discriminant}
\label{ssec:discriminantreduction}

Given a
binary form $F(X,Z)\in k[X,Z]$ of any degree~$n\geq 3$, we wish to
find a $\GL_2(k)\times k^*$-equivalent form with minimal discriminant.
First we recall that the 
discriminant
of a separable binary form
$$F(X,Z) = \prod_{i=1}^{n} (\gamma_i X - \alpha_i Z)\in k[X,Z]$$
with $\alpha_i,\gamma_i\in\kbar$
is $$\Delta(F) = \prod_{i<j} (\gamma_j\alpha_i-\gamma_i\alpha_j)^2\in k^*.$$

In terms of the polynomial $f=F(x,1)$ of degree $n$ or~$n-1$
with leading coefficient~$c$,
this is
$$\Delta(F) = \left\{ \begin{array}{rl}
 \Delta(f)&\qquad \mbox{if }\deg f = n,\\
 c^2 \Delta(f) &\qquad \mbox{if }\deg f = n-1.
\end{array}\right.$$

Let $g\in\ZZ$ be given by $n=2g+2$ if $n$ is even
and $n=2g+3$ if $n$ is odd. 
If~$n$ is even and~$\geq 6$, then~$F$ corresponds to a hyperelliptic
curve $C_F$ of genus $g$ with
$$\Delta(C_F) = 2^{4g} \Delta(F).$$
If $n$ is odd, then there is no interpretation in terms
of hyperelliptic curves and the number~$g$ is
simply a convenient number in the algorithms and proofs.

The discriminant changes under the action of the group $\GL_2(k)\times k^*$ via
\begin{equation}
\label{eq:disc}
\Delta(F\cdot [A,u]) = u^{2(n-1)} \det(A)^{n(n-1)} \Delta(F).
\end{equation}

\begin{remark}\label{rem:disc}
In case $n=6$, the Igusa invariants
of Section~\ref{ssec:icp} satisfy
$I_{10}(C) = 2^{12} \Delta(C) = 2^{20}\Delta(F)$ and
$$I_{j}(C_{F\cdot [A,u]}) = u^{j} \det(A)^{3j} I_{j}(C_F).$$
\end{remark}

Before we describe how to reduce the discriminant
globally over a number field,
we first describe how to reduce the discriminant
at just one prime.

\subsubsection{Local reduction of the discriminant}
\label{ssec:reductiondiscriminantlocal}

Assume for now that $k$ is the field of fractions
of a discrete valuation ring $R$
with valuation~$v$. Let $\pi$ be a uniformiser of~$v$
and $\mathfrak{m}=\pi R$ the maximal ideal.

We call $F$ \emph{minimal at $v$}
if $v(\Delta(F))$ is minimal among
all $\GL_2(k)\times k^*$-equivalent forms
with $v$-integral coefficients.

\begin{proposition}\label{prop:local}
Suppose $F\in H_n(k)$ has coefficients in~$R$.
Let $g=\lfloor n/2\rfloor -1$ be the largest integer smaller than or equal to
$(n-2)/2$, so $n\in \{2g+2, 2g+3\}$.
Then $F$ is \emph{non}-minimal at~$v$
if and only if
we are in one of the following three cases.
\begin{enumerate}
\item The polynomial $F$ is not primitive,
so $\fnew = F\cdot[\mathrm{id}_2, \pi^{-1}]$ is integral
and satisfies $v(\Delta({\fnew}))<v(\Delta(F))$.
\item The polynomial $(F(x,1)\ \text{mod}\ \mathfrak{m})$
has a $(g+2)$-fold root~$\overline{t}$ in the residue field.
Moreover, for \emph{some} (equivalently \emph{every})
lift $t\in R$ of~$\overline{t}$,
the form $\fnew = F\cdot [({\pi\atop 0} {t\atop 1}), \pi^{-(g+2)}]
=F(\pi X+tZ, Z)\pi^{-(g+2)}$ is integral
and satisfies $v(\Delta(\fnew)) < v(\Delta(F))$.
\item The polynomial $(F(x,1)\ \text{mod}\ \mathfrak{m})$
has degree~$\leq n-(g+2)$.
Moreover, the form $\fnew = F\cdot [({1\atop 0} { 0\atop  \pi}), \pi^{-(g+2)}] = \pi^{-(g+2)}F(X, \pi Z)$ is integral
and satisfies $v(\Delta({\fnew})) < v(\Delta(F))$.
\end{enumerate}
\end{proposition}

\begin{proof}
\nocite{liu-local-hyperelliptic}
For the ``if'' part, note that 
in each of the three cases, the proposition
gives an explicit equivalent form
that proves that~$F$ is not minimal.

Conversely, suppose that~$F$ is non-minimal.
Then there exists $[A,u]\in\GL_2(k)\times k^*$ with $F\cdot [A,u]$ integral
of smaller discriminant.
Write \[A = \left(\begin{array}{cc} a & b \\ c & d\end{array}\right).\]

Let $T$ be the subgroup $T=\{[\mu \mathrm{id}_2, \mu^{-n}] : \mu\in k^*\}$ of the centre of $\GL_2(k)\times k^*$,
and note that~$T$
acts trivially on~$H_n(k)$, so without loss of generality $A$ has coprime coefficients in~$R$,
so either (i) $c\in R^*$ or $d\in R^*$ or (ii) $c\equiv d\equiv 0\ \mathrm{mod}\ \pi$ and $a$ or $b$ is in $R^*$.

Note also that $\GL(R)\times R^*$ preserves integrality and the discriminant,
so we use multiplication by $\GL(R)$ on the right to perform elementary column operations over~$R$ on~$A$.
We get that without loss of generality either
(i) $d=1$, $c=0$ or (ii) $a=1$, $b=0$, $c\equiv d\equiv 0\ \mathrm{mod}\ \pi$.

Note that in both cases $a\not=0$ and $c\not=0$, so
with more $\GL(R)\times R^*$-multiplication, we get $a=\pi^k$, $d=\pi^l$, $u=\pi^{-m}$
with $k$, $l$, $m \in\ZZ$, $k$, $l\geq 0$, and by equation~\eqref{eq:disc} also
\begin{equation}
\label{eq:disc2}
2m>n(k+l).
\end{equation}

We start with case (i). 

Let $H(X,Z) = F(X+bZ, Z)$ and write $H(X,Z) = \sum_i h_i X^iZ^{n-i}$.
Then $F\cdot [A,u] = \pi^{-m} H(\pi^k X,Z)$ is integral, so
$v(h_i) \geq m - ki.$
Together with~\eqref{eq:disc2}, this gives
$$v(h_i) > \left(\frac{n}{2}-i\right)k.$$
In particular, if $k=0$, then~$H$ is integral and non-primitive,
hence so is~$F(X,Z) = H(X-bZ,Z)$ and we are in case~1.

\newcommand{\smalltbt}[4]{({#1 \atop #3}{#2 \atop #4})}
\newcommand{\tbt}[4]{\left(\begin{array}{cc} #1 & #2 \\ #3 & #4\end{array}\right)}
If $k\geq 1$, then for all $i$, we have
$v(h_i)> \frac{n}{2}-i$, hence $v(h_i) > \lfloor n/2\rfloor -i = g+1-i$,
so $v(h_i)\geq g+2-i$.
In particular, the form
$F\cdot [\smalltbt{\pi}{b}{0}{1}, \pi^{-(g+2)}] =
 H\cdot [\smalltbt{\pi}{0}{0}{1}, \pi^{-(g+2)}]$
is integral, and of strictly smaller discriminant than~$F$.
This proves that we are in case~2 for \emph{some} lift $t=b$
of a $(g+2)$-fold root $\overline{t}=\overline{b}$.
To finish the proof of case~2, we need to prove that
for \emph{every} $t'$ satisfying $\overline{t'}=\overline{b}$,
the transformation
$[\smalltbt{\pi}{t'}{0}{1}, \pi^{-(g+2)}]$
also gives an integral equation.

Let $y = (t'-b)/\pi\in\mathcal{O}_k$ and note
$$\tbt{\pi}{t'}{0}{1} = \tbt{\pi}{b}{0}{1} \tbt{1}{y}{0}{1}
\in
\tbt{\pi}{b}{0}{1}\GL_2(R)
,$$
which proves that we are in case~2 for \emph{every} lift~$t$.
This finishes case (i).

Now assume that we are in case~(ii).
Equation~\eqref{eq:disc2} gives $m>\frac{n}{2}\geq g+1$.

Write $F=\sum_{i=0}^{n} f_i X^iZ^{n-i}$.
We will prove by induction that $v(f_j) \geq j + g + 2 - n$ holds
for all~$j$, which implies that
$F(X,\pi Z)\pi^{-(g+2)}$ is integral,
so we are in case~3.
Note that the assertion is
trivial for $j\leq n-g-2$. Now suppose that it is true for all $j<J$.

Note that $F\cdot [A,u] = \pi^{-m} F(X, cX+dZ)$ is integral,
so modulo $\pi^{g+2}$, we get
$0\equiv \sum_{i=0}^{n} f_i X^i (cX+dZ)^{n-i} \equiv \sum_{i=J}^{n} f_i X^i (cX+dZ)^{n-i}$.
Looking at the coefficient of $X^JZ^{n-J}$, we get
$f_J d^{n-J}\equiv 0\ \mathrm{mod}\ \pi^{g+2}$, so
$\pi^{g+2-n+J}\mid f_J$.
This finishes the proof.
\end{proof}

We use Proposition~\ref{prop:local} to create the following reduction algorithm.

\begin{algorithm}[Local Reduction]\label{alg:loc}~\\
\textbf{Input}: A binary form $F\in H_n(k)\cap R[X,Y]$ and a prime element $\pi\in R$.\\
\textbf{Output}: A binary form $\fnew$ that is $\GL_2(k)\times k^*$-equivalent and minimal
at $\mathrm{ord}_\pi$.\\
First let $g=\lfloor n/2\rfloor -1$.
\begin{enumerate}
\item If $F\ \text{mod}\ \pi R$
is zero, then repeat the algorithm with $\fnew = F\cdot[\mathrm{id}_2, \pi^{-1}]$.
(This corresponds to case~\ref{prop:local}.1.)
\item If  $F(x,1)\ \text{mod}\ \pi R$ has degree~$\leq n-(g+2)$,
then let $\fnew = F(X, \pi Z)\pi^{-(g+2)}$.
If $\fnew$ is integral, then repeat the algorithm with $\fnew$.
(This corresponds to case~\ref{prop:local}.3.)
\item
Factor $\overline{f}=(f\ \mathrm{mod}\ \pi)$ over the finite $R/\pi R$.
If $\overline{f}$ has a root $\overline{t}$
of multiplicity $\geq g+2$,
then let $t$ be a lift of $\overline{t}$ to~$R$.
If $F^\dagger = F(\pi X+tZ, Z)\pi^{-(g+2)}$
is integral, then repeat the algorithm with~$\fnew$.
(This corresponds to case~\ref{prop:local}.2.)
\item Return $F$.
\end{enumerate}
\end{algorithm}

\begin{proof}[Proof of correctness of Algorithm~\ref{alg:loc}]
Every step of the algorithm leaves the model integral,
and every iteration reduces $v(\Delta(F))$,
so the algorithm terminates.
It therefore suffices to prove that the output
is not in any of the three cases of Proposition~\ref{prop:local}.

In case~1, the algorithm reduces the discriminant in step~1
and starts over. In case~3, the same happens with step~2,
and in case~2, it happens with step~3
because a polynomial of degree $\leq  2g+3$
has at most one $(g+2)$-fold root~$\overline{t}$.
\end{proof}
In many cases, we can do step~3 as follows without having to think
about factoring of polynomials.
\begin{lemma}\label{lem:changestep3}
If $\pi$ is coprime to~$n!$, then
step~3 can be replaced by the following.
\begin{itemize}
\item[3'.] Let $f=F(x,1)$, calculate $\gcd(f,f',f'',\dots,f^{(g+1)})$ over the finite field $R/\pi R$,
and write it as $\sum_{i=0}^s a_i x^s$ with $a_s\not=0$.
 If $s>0$, then let $t$ be such that $t\equiv -a_{s-1}/(sa_s )\ \text{mod}~\pi R$.
If $F^\dagger = F(\pi X+tZ, Z)\pi^{-(g+2)}$
is integral, then repeat the algorithm with~$\fnew$.
\end{itemize}
\end{lemma}
\begin{proof}
It suffices to show that if $\overline{f}$ has a root $\overline{t}$
of multiplicity $\geq g+2$,
then it is equal to $(-a_{s-1}/(sa_s)\ \mathrm{mod}\ \pi R)$.

Let $a$ be a root of exact multiplicity~$m$ of~$\overline{f}$ over
the algebraic closure of $R/\pi R$,
that is, we have $\overline{f} = (x-a)^m g(x)$ with
$g(a)\not=0$.
Then the $i$-th derivative $\overline{f}^{(i)}$ for $i\leq m$
is $$\frac{m!}{(m-i)!} (x-a)^{m-i}g(x)\qquad \mbox{modulo}\qquad(x-a)^{m-i+1}.$$
In particular, $(x-a)$ is a factor of 
$\gcd(\overline{f}, \overline{f}',\ldots, \overline{f}^{(m-1)})$,
but not of $\overline{f}^{(m)}$.
Here we use that $m!$ is coprime to~$\pi$.

It follows that only the (unique) root of multiplicity
$\geq g+2$ appears in $\gcd(\overline{f},\overline{f}',\ldots, \overline{f}^{(g+1)})$,
that is, we get $\overline{f} = a_s (x-\overline{t})^s$,
hence $a_{s-1} = -s\overline{t}a_s$, so
$\overline{t} = -a_{s-1}/(sa_s)$.
\end{proof}

\subsubsection{Global reduction of the discriminant}
\label{ssec:reductiondiscriminantglobal}

Now let us get back to the case where $k$ is a number field with
ring of integers $\mathcal{O}_k$. We prefer to have a binary form $F$
where $v(\Delta(F))$ is minimal for all discrete valuations $v$ of $k$.

If $k$ has class number one, then such a form exists. 
Indeed, if we take~$\pi$ in Algorithm~\ref{alg:loc} to be a generator
 of the prime ideal corresponding to~$v$, then this affects only~$v$
 and no other valuations, so we can do this for each~$v$ separately.
See Section~\ref{ssec:classnumber}
for what to do if the class group is non-trivial.

To be able to use our local reduction algorithm one prime at a time,
we need to know
the valuations $v$ for which $v(\Delta(F))$ is non-minimal. The most
straightforward method is to factor~$\Delta(F)$.
However, factorisation is computationally hard, so we will give some tricks
for trying to avoid factorisation below.
We needed to use a combination of sophisticated factorisation software and
the tricks below for creating our tables. Indeed, on the one hand,
without the tricks below,
even the state-of-the-art factorisation software left us unable to
reduce a couple of the curves.
On the other hand, when just using the tricks below and the built-in
factorisation functionality of pari-gp~\cite{pari} (through Sage~\cite{Sage}),
there are some curves that we were still unable to reduce.
Only the combination of factoring software and the tricks below allowed
us to complete the table.

For serious factoring, we combined the built-in
implementation of Pollard's rho method
and the elliptic curve method of Magma~\cite{magma},
the GMP-ECM implementation of the elliptic curve method~\cite{GMPECM}, 
and the CADO-NFS implementation of the number field sieve~\cite{CADONFS}.

The method for avoiding factorisation is based on the following fact.

\begin{proposition}\label{prop:alg}
Let $\mathfrak{a} = \pi\mathcal{O}_{k}$ be any (possibly non-prime) principal ideal
in a number field~$k$.
Modify Algorithm~\ref{alg:loc} as follows.
\begin{enumerate}
\item Whenever testing whether an element
$b$ of $\mathcal{O}_k$ is zero modulo $\pi^j \mathcal{O}_k=\mathfrak{a}^j$
or whether an element $b/\pi^j\in k$ is integral
(in Steps 1, 2, and~3), compute
$\mathfrak{d}_i = \gcd(b\mathcal{O}_k, \mathfrak{a}^i)$ for $i=1,\ldots j-1$.
If there exists an $i$ with $\mathfrak{d}_i\not\in\{\mathfrak{a}^{i-1}, \mathfrak{a}^i\}$,
then for the smallest such $i$ output the non-trivial factor $\mathfrak{d}_i/\mathfrak{a}^{i-1}$ of~$\mathfrak{a}$.
\item Replace step~3 with step~3' of Lemma~\ref{lem:changestep3}
regardless of whether $\pi$ is coprime to $n!$.
Compute $\gcd$s of polynomials in $\mathcal{O}_k/\mathfrak{a}$ using Euclid's algorithm. For each division
with remainder by a polynomial $g$, first compute the gcd of the leading coefficient of $g$ with $\mathfrak{a}$
as in item~1.
\end{enumerate}
Then all steps of Algorithm~\ref{alg:loc} are
 polynomial-time computable and the output is either a polynomial
$\fnew$ equivalent to~$F$ with $\Delta(\fnew)\mid \Delta(F)$ 
or a non-trivial factor of~$\mathfrak{a}$.
Moreover, if~$\mathfrak{a}$ is 
square-free and coprime to $n!$
and the algorithm runs without returning
a factor of~$\mathfrak{a}$, then the output polynomial~$\fnew$
is minimal at all primes dividing~$\mathfrak{a}$.
\end{proposition}
\begin{proof}
Since the leading coefficient of a polynomial over $\mathcal{O}_k$ is either invertible
modulo $\mathfrak{a}$ or has a non-trivial factor in common with $\mathfrak{a}$,
division with remainder either works or provides such a non-trivial factor.
This proves the first assertion in Proposition~\ref{prop:alg}.

Next suppose that~$\mathfrak{a}$ is square-free
and coprime to~$n!$
and let~$F$ be as in Algorithm~\ref{alg:loc}.
If~$F$ is minimal at all primes dividing~$\mathfrak{a}$,
then we are done.
If there is an $i\in\{1,2,3\}$ such that
all primes dividing~$\mathfrak{a}$
are as in Proposition~\ref{prop:local}.$i$,
then the corresponding step (1, 3' or 2) in Algorithm~\ref{alg:loc} reduces the discriminant of~$F$
and we start over with a new~$F$.

So without loss of generality,
there are $i\in\{1,2,3\}$ and primes $\mathfrak{p},\mathfrak{q}\mid \mathfrak{a}$
such that $\mathfrak{p}$ is as in Proposition~\ref{prop:local}.$i$ and $\mathfrak{q}$ is not.
But then the corresponding step (1, 3' or 2)
in Algorithm~\ref{alg:loc}
returns a non-trivial factor of~$\mathfrak{a}$.
\end{proof}

Based on Proposition~\ref{prop:alg}, we get the following algorithm that tries to minimise
the amount of factoring.

\begin{algorithm}\label{alg:glo}~\\
\textbf{Input}: A binary form $F\in H_n(k)$ for a number field $k$ of class number one.\\
\textbf{Output}: A binary form  $\fnew$ 
that is integral, is $\GL_2(k)\times k^*$-equivalent to $F$,
and has minimal discriminant.
\begin{enumerate}
\item Let $\mathfrak{a} = \Delta(F)\mathcal{O}_k$ and $A=\{\mathfrak{a}\}$.
\item If the unit ideal is in $A$, remove it from $A$.
If $A$ is empty, return~$F$.
\item For each $\mathfrak{a}\in A$, test if $\mathfrak{a}$ is a perfect power and replace
it by its highest-power root.
\item Fix $B\in\ZZ$ with $B\geq n$
and apply trial division up to $B$ to each element of~$A$
to find a small prime factor $\mathfrak{p}=(\pi)$.
If no prime is found, go to Step~5.
If a prime is found, 
then reduce the form
locally using Algorithm~\ref{alg:loc} on $\mathfrak{p}$, remove all factors~$\mathfrak{p}$
from all elements of~$A$, and go to step~2.
\item For each $\mathfrak{a}\in A$, run Algorithm~\ref{alg:loc} on $\mathfrak{a}$ 
with the modifications of Proposition~\ref{prop:alg}.
\begin{enumerate}
\item If it returns a non-trivial factor~$\mathfrak{b}$
of~$\mathfrak{a}$, then replace $\mathfrak{a}$ in $A$ by $\mathfrak{b}$ and $\mathfrak{a}/\mathfrak{b}$
and go to step~3.
\item If it returns a binary form $F^\dagger\not=F$,
then replace all $\mathfrak{a}\in A$ by
$\mathfrak{a}+\Delta({F^\dagger})\mathcal{O}_k$, replace $F$ by $F^\dagger$,
and go to step 2.
\item If it returns $F$,
then go to the next $\mathfrak{a}$ in~$A$.
\end{enumerate}
\item Go to step~4 with a strictly larger trial division bound~$B$
(or more sophisticated factoring methods).
\end{enumerate}
\end{algorithm}
Let us first show that this algorithm terminates in finite time and returns
a minimal form.
For minimality of the form, note that at every step in the algorithm, all primes at which~$F$ is non-minimal
divide some element of~$A$, and the algorithm terminates only if $A$ is empty.
To see that the algorithm ends, note that the norm $N = N_{k/\QQ}(\Delta(F))$ never increases,
while at every iteration either $N\in\ZZ$ decreases or $B\in\ZZ$ increases,
so at some point we have $B>N$ after which a repeated application of step~4 finishes the algorithm.

\begin{remark}
There is no way to completely avoid factoring. Indeed, if one can compute the twist-minimal
model of the hyperelliptic curve 
$$y^2 = N^2 x^6+x+1\quad\mbox{where $N=pq^2$ with $p,q$ prime},$$
then one can also factor the integer $N=pq^2$.
\end{remark}

\begin{remark}
In the genus-two case (that is, $n=6$) we can replace $\Delta(F)\mathcal{O}_k$ in the algorithm
by the ideal $\gcd(I_2(C_F),I_4(C_F),I_6(C_F),\Delta(F))$,
where $I_2, I_4, I_6$ are the Igusa-Clebsch invariants from Section~\ref{sec:invcmmain}.
Indeed, we have that $I_2$, $I_4$, and~$I_6$ satisfy the transformation formula of Remark~\ref{rem:disc},
so all primes at which the model is non-minimal divide this gcd.
The advantage is that this ideal is smaller than~$\Delta(F)$, which speeds up the algorithm.
\end{remark}

\begin{remark}
All of the above works if one wants a hyperelliptic curve
model that is 
isomorphic over $\overline{k}$, but not necessarily
over~$k$.
To get a minimal model of $C_F$ that is isomorphic over~$k$,
one could do the following.
First reduce~$F$ as above, and do some bookkeeping to find not only a
twist-reduced model $C_{F^\dagger}/k$, but also $[A,u]\in\GL_2(k)\times k^*$
with $F^\dagger = F\cdot [A,u]$ and some information on the factorisation of~$u$.
Then all one needs is a minimal element $v\in u(k^*)^2\cap \mathcal{O}_k$,
because $C_{vF^\dagger}$ is then a minimal model.
Such an element $v$ exists if $k$ has class number one, and can then be found
easily if one is able to factor~$u\mathcal{O}_k$.
\end{remark}

\subsubsection{Class number $>1$}
\label{ssec:classnumber}

Everything in Section~\ref{ssec:reductiondiscriminantglobal}
was under the assumption that $k$ had class number one, and hence a global
minimal form exists. If $k$ does not have class number one, then this is not always possible.
 Indeed, let $F_v$ be a $\GL_2(k)\times k^*$-equivalent binary
form with $v(\Delta(F_v))$ minimal, and let 
$\Delta_\mathrm{min}$ be the ideal with $v(\Delta_{\mathrm{min}})=v(\Delta(F_v))$ for all~$v$. 
If $\Delta_{\mathrm{min}}$ is not principal, then there is no form
 with that discriminant. In fact, if $F$ is any form, and there exists 
a globally minimal equivalent form $F_{\mathrm{min}}$ with 
$\Delta(F_{\mathrm{min}})=\Delta_{\mathrm{min}}$, then the ideal 
$\sqrt[\gcd(n,2)(n-1)]{\Delta(F)/\Delta_{\mathrm{min}}}$ is a principal ideal.

So instead of a globally reduced form, we look for an almost-reduced form.
Let $S$ be a (small) set of (small) prime ideals that generate the class group.
It is easy to change the methods above into an algorithm that finds a
form that is reduced outside~$S$. We now give the details of the algorithm that we used for this,
which also makes the form reasonably simple at the primes of~$S$.

Let $T$ be any set of prime ideals that generate the class group
and $\mathfrak{a}$ an ideal supported outside~$T$.
In Algorithm~\ref{alg:loc}, to reduce at~$\mathfrak{a}$ and stay reduced outside of~$T$, we
do the following.
Take $\pi_u\in \mathfrak{a}$
and $\pi_l^{-1}\in\mathfrak{a}^{-1}$ such that $\pi_u/\mathfrak{a}$ and
$\mathfrak{a}/\pi_l$
are supported on~$T$.
Then in Algorithm~\ref{alg:loc} replace the formulas
for $\fnew$ in cases 1, 2, 3 with
\begin{equation}\label{eq:nonmin}
\pi_l^{-1}F(X,Z),\quad  F(X/\pi_l, Z)\pi_u^{n-(g+2)}, \quad\mbox{and}\quad F(\pi_u X+t Z, Z)\pi_l^{-(g+2)}
\end{equation}
respectively, 
where we make sure that $t$ is divisible by~$\pi_u/\mathfrak{a}$.
Note that this gives integral forms, and worsens the discriminant only at~$T$.

Our algorithm starts by taking $T$ disjoint from~$S$.
First reduce at all primes of~$S$, possibly worsening at~$T$.
Then take $T=S$ and reduce outside of~$S$, possibly worsening at~$S$.

Since we had a minimal form at the primes of~$S$,
the only non-minimality of the form at this stage is 
what was introduced by~\eqref{eq:nonmin}.
In particular, it can be removed by transformations of the form~$a^{-1}b^g F(b^{-1}X,Z)$.
So we take $a,b\in\mathcal{O}_k$ with $a^2b^{n-2g}$ of maximal norm such that $a^{-1}b^g F(b^{-1}X,Z)$ is integral.
Note that no hard factoring is required in finding $a$ and $b$ since they are
supported on the set of primes~$S$.

We did the above for the field $K=\QQ[X] / (X^4+46X^2+257)$
(denoted $[17,46,257]$ in~\cite{echidna}).
We used $S=\{\mathfrak{p}\}$ for a (non-principal) prime $\mathfrak{p}$ of norm $2$ in
the quadratic field $K_0^r=\QQ(\sqrt{257})$, which has class group
of order~$3$.

\subsection{Reduction of coefficients: Stoll-Cremona reduction}
\label{ssec:stollcremona}

At this point, we have an integral form~$F\in H_n(k)$ 
where the norm $N(\Delta(F))$ is small. 
Next, we try to make the coefficients small.
As we do not want to break integrality or disturb the discriminant,
we take transformations in $(\GL_2(\mathcal{O}_k)\times\mathcal{O}_k^*)$.

We use a notion of `reduced' based on Stoll
and Cremona~\cite{stoll-cremona}. We do not prove
that this notion of `reduced' yields small coefficients,
but in practice it does.

\subsubsection{The case $k=\QQ$}

Stoll and Cremona~\cite[Definition 4.3]{stoll-cremona} give a definition of \emph{reduced} 
for binary forms of degree $\geq 3$ over~$\QQ$ under the 
action of~$\SL_2(\ZZ)\times 1$, which we will summarise here.

Recall that $H_n(k)$ is the set of
separable binary forms $F(X,Y)$ of degree~$n$.
Let $\mathcal{H}=\{z\in\CC : \mathrm{Im}(z)>0\}$ be the complex upper half plane. 
We turn the standard left $\GL_2(\RR)^{+}$-action on $\mathcal{H}$ into a right action 
by $$z\cdot A = A^{-1}(z) = \frac{dz-b}{-cz+a}$$ for $A=({a \atop c}{b\atop d})$. 

The idea behind~\cite{stoll-cremona} is to use an $\SL_2(\RR)$-covariant map $z:H_n(\RR)\rightarrow \mathcal{H}$. 
In~$\mathcal{H}$, there is a notion of $\SL_2(\ZZ)$-reduction, 
and we just pull back that notion to $H_n(\QQ)$ via~$z$. 
In other words, we have the following definition.
\begin{definition}\label{def:reduced}
We call $F\in H_n(\QQ)$ \emph{reduced for $\SL_2(\ZZ)$} 
if $z(F) = z = x+iy$ satisfies
\begin{itemize}
\item[(R)] $|x|\leq \frac{1}{2}$, and
\item[(M)] $|z|\geq 1$.
\end{itemize}
\end{definition}
This gives rise to the following algorithm.
\begin{algorithm} (Stoll-Cremona reduction)\\
\textbf{Input:} $F\in H_n(\QQ)$\\
\textbf{Output:} an $\SL_2(\ZZ)$-reduced
element of the orbit $F\cdot (\SL_2(\ZZ)\times 1)$.
\begin{enumerate}
\item Let $m$ be the integer nearest to $x=\mathrm{Re}(z(F))$ and let
$F\leftarrow F\cdot({1\atop 0}{m\atop 1}) = F(X+mZ,Z)$.\\
This replaces $z(F)$ with $({1\atop 0}{-m\atop \phantom{-}1})z(F) = z(F)-m$, which satisfies (R) above.
\item If $|z(F)|<1$,
then let
$F\leftarrow F\cdot ({\phantom{-}0\atop -1}{1\atop 0})=F(Z,-X)$ and go back to step~1. This replaces $z(F)$ with $({0\atop 1}{-1\atop 0})z(F) = -1/z(F)$,
which satisfies (M) above.
\end{enumerate}
\end{algorithm}

Stoll and Cremona~\cite[after Proposition 4.4]{stoll-cremona} outline how one could
extend the definition of reduced to binary forms
over any number field $k$ under the action of~$\SL_2(\mathcal{O}_k)\times 1$.
We work out the details in the case of 
a totally real field,
and give an implementation and an improvement.

To generalise the algorithm, we need two ingredients: a covariant map,
and a reduction algorithm on the codomain of that map.

\subsubsection{The covariant for totally real fields}

Let~$k$ be a totally real number field of degree~$d$ and
 let $\phi_1,\ldots,\phi_d$ be the $d$ embeddings $k\rightarrow \RR$. 
This induces embeddings $k\rightarrow \RR^d$, $H_n(k)\rightarrow H_n(\RR)^d$ and $\SL_2(k)\rightarrow \SL_2(\RR)^d$,
which we will use implicitly.
Composing with the covariant map~$z$ on every component, we get 
a map $H_n(k)\rightarrow \HH^d$, which is $\SL_2(k)$-covariant
and which we also denote by~$z$.
\begin{remark}
The quotient space $\SL_2(\mathcal{O}_k)\backslash \HH^d$ is coincidentally
the Hilbert moduli space of polarised
abelian $d$-folds with real multiplication by~$\mathcal{O}_k$
and a certain polarisation type.
\end{remark}

In fact, we can do slightly better. We identify $\mathcal{H}$ with 
$(\CC\setminus \RR)$ modulo complex conjugation, that is, 
we identify $z\in-\mathcal{H}$ with $\overline{z}\in\mathcal{H}$.
Then the $\SL_2(\RR)$-action on $\mathcal{H}$ extends to 
a $\GL_2(\RR)$-action also given by 
$z\cdot A = A^{-1}(z) = (dz-b)/(-cz+a)$ (up to complex conjugation). 
The covariant~$z$ of \cite{stoll-cremona} then turns out to also be $\GL_2(\RR)$-covariant.
In particular, we get a map
\[z: H_n(k)\rightarrow \HH^d,\quad\mbox{ which is $\GL_2(k)$-covariant.}\]

\subsubsection{Reduction for $\GL_2(\mathcal{O}_k)$ in $\HH^d$}

\newcommand{\ourgroup}{\GL_2(\mathcal{O}_k)} 

Let $N:\RR^d\rightarrow \RR: (x_m)_m\mapsto \prod_m x_m$,
define $\mathrm{Re}, \mathrm{Im}, |\cdot| : \CC^d\rightarrow \RR^d$
component-wise
and let $\log:\RR^d\rightarrow \RR^d:(x_m)_m\mapsto (\log|x_m|)_m$.

\begin{definition}\label{def:reducedhigher}
We call $z \in \HH^d$ \emph{reduced}
for $\ourgroup$ if it satisfies the following conditions.
\begin{itemize} 
\item[(R)] The point $\mathrm{Re}(z)\in \RR^d$ is in some fixed chosen fundamental hyper-parallelogram
 for addition by~$\mathcal{O}_k$,
\item[(I)] the point $\log (\mathrm{Im}(z))\in \RR^d$ is in some fixed chosen fundamental 
domain for addition by~$\log(\mathcal{O}_k^*)$, and
\item[(M)] the norm $N(\mathrm{Im}(z))$ is maximal for the $\ourgroup$-orbit
$\ourgroup z$.
\end{itemize}
\end{definition}
Let us first see how this is an analogue of Definition~\ref{def:reduced}.
Note that in the case $k=\QQ$,
we can choose the hyper-parallelogram $[-\frac{1}{2},\frac{1}{2}]$,
and then
conditions \ref{def:reduced}(R) and~\ref{def:reducedhigher}(R)
coincide and condition \ref{def:reducedhigher}(I) is empty.
It is well-known that under condition \ref{def:reduced}(R), we have \ref{def:reducedhigher}(M)
if and only if~\ref{def:reduced}(M).

\begin{remark}Definition~\ref{def:reducedhigher} is also closely related to a standard
definition of \emph{reduced} for the action of $\SL_2(\mathcal{O}_k)$ on~$\HH^d$.
Indeed, if $k$ has class number one and we replace $\ourgroup$ with $\SL_2(\mathcal{O}_k)$
and $\mathcal{O}_k^*$ with $(\mathcal{O}_k^*)^2$, then we get a fundamental domain of~\cite{vdgeer}.
One could use the standard fundamental domain from~\cite{vdgeer} in general, 
but since we had only one case of class number $>1$, we simply used (R), (I) and (M) for that field as well.
\end{remark}

The above gives rise to a notion of reduction for $\GL_2(\mathcal{O}_k)\times 1$ on $H_n(k)$.
We then get the following sketch of a reduction algorithm.

\begin{algorithmsketch}[{Reduction for $\GL_2(\mathcal{O}_k)\times \mathcal{O}_k^*$)\label{alg:sketch}}]\hfill\\
\textbf{Input:} $F\in H_n(k)$.\\
\textbf{Output:} $\fnew \in H_n(k)$ that is $\GL_2(\mathcal{O}_k)\times \mathcal{O}_k^*$-equivalent to
$F$ and $\GL_2(\mathcal{O}_k)$-reduced.
\begin{enumerate}
\item Compute a fundamental domain $\mathcal{F}$ for addition by $\mathcal{O}_k$ in $\RR^d$.
\item Compute a fundamental domain $\mathcal{G}$ for addition by 
$\log(\mathcal{O}_k^*)$
in $\RR^d$.
\item Take $u\in\mathcal{O}_k^*$ such that $\log \mathrm{Im}(z(F)) - (\log |\phi_m(u)|)_m \in \mathcal{G}$
and replace $F$ by
\[ F \cdot \left[\left(\begin{array}{rr} u & 0 \\ 0 & 1 \end{array}\right), 1\right] = F(uX, Z).\]
This replaces $z(F)$ by $u^{-1} z(F)$, hence
makes sure $F$ satisfies~(I) and preserves $N(\mathrm{Im}(z))$.
\item Take $b\in\mathcal{O}_k$ such that $\mathrm{Re}(z(F))-b\in\mathcal{F}$ and replace $F$ by
\[ F \cdot \left[\left(\begin{array}{rr} 1 & b \\ 0 & \phantom{-}1 \end{array}\right), 1\right] = F(X+bZ, Z).\]
This replaces $z(F)$ by $z(F)-b$, hence
makes sure $F$ satisfies (R) and preserves (I) and~$N(\mathrm{Im}(z))$.
\item Try to find a matrix $M$ such that $N(\mathrm{Im}(Mz))>N(\mathrm{Im}(z))$.
If no such matrix exists, go to step~6.
If such a matrix exists, replace $F$ by $F \cdot [M^{-1}, 1]$ and go to step~3.
\item Try to find $u\in\mathcal{O}_k^*$ such that the maximum of the heights
of the coefficients of $uF$ is minimal
and return $F [1_2, u] = uF$.
\end{enumerate}
\end{algorithmsketch}

Bases of $\mathcal{O}_K$ and $\mathcal{O}_K^*$
are easy to compute using a number theory package
like Magma~\cite{magma} or Pari~\cite{pari},
hence so are $\mathcal{F}$ and~$\mathcal{G}$.
Numerical approximation of the covariant $z:H_n(\RR)\rightarrow \HH$
of~\cite{stoll-cremona}
is available in Magma as a 
standard function (called \texttt{Covariant}).
So the only steps with missing details are 
5 and~6.

For step 5, note first that for $A
=({a\atop c}{b\atop d})
\in \GL_2(\RR)$
and $z\in \HH$
we have
$\mathrm{Im} Az =
|\det A|\cdot |cz+d|^{-2} \cdot \mathrm{Im}(z)$.
In particular, for $A\in \GL_2(\mathcal{O}_k)$
and $z\in \HH^g$, we have
$N(\mathrm{Im} Az) = N(|cz+d|)^{-2} N(\mathrm{Im} z)$
so the condition in step~5 is equivalent
to $N(|cz+d|)< 1$.
Given $c, d\in\mathcal{O}_k$, it is easy to find $a, b\in\mathcal{O}_k$
with $ad-bc\in\mathcal{O}_k^*$ if they exist, so for step~5, we need
only to find
$c, d$.

A fast first attempt at trying to find $c, d$ for step~5 is to consider the lattice
$\{cz+d\in\CC^g : c,d\in \mathcal{O}_k\}$ and compute
an LLL-reduced $\ZZ$-basis. If the first vector $cz+d$ of the LLL-reduced
basis satisfies $N(|cz+d|)<1$, then use these $c$ and~$d$.
Note that this always works if the covolume $N(\mathrm{Im}(z))$
of the lattice
is sufficiently small.

If the first attempt for step~5 fails, then we use an exhaustive search
as follows.
Note first of all that we only need to consider pairs $(c,d)$
up to multiplication by~$\mathcal{O}_k^*$.
Note $$N(|cz+d|)^2 \geq N(|c\mathrm{Re}(z)+d|)^2 + |N(c)|^2 \cdot N(|\mathrm{Im}(z)|)^2,$$
so the $c,d$ that we need satisfy $|N(c)| < N(|\mathrm{Im}(z)|)^{-1}$ and $c\not=0$.
We list all such~$c$ up to multiplication by units
by listing all ideals of norm $< N(|\mathrm{Im}(z)|)^{-1}$.
Next, the numbers $d$ that we need satisfy
$$\sqrt{1 - |N(c)|^2 \cdot N(|\mathrm{Im}(z)|)^2} > N(|c\mathrm{Re}(z) +d|)
\geq N(|c\mathrm{Re}(z)|) + \sum_i \left(\prod_{j\not=i} |\phi_j(c)\mathrm{Re}(z_j)|\right) |\phi_i(d)|,$$
which yields a bounded box in $\RR^d$ containing~$d$.
So we list all $d\in\mathcal{O}_k$ in that box.
This exhaustive search is guaranteed to find all relevant $c,d$, after which we choose
the pair $c,d$ with minimal $N(|cz+d|)$.
Since we have the minimal $N(|cz+d|)$, we also have the maximal
$N(|\mathrm{Im}(Az)|)$ for the whole orbit,
hence the algorithm finishes after one more iteration of steps 3, 4, 6.
The exhaustive search for $c,d$ can however be very slow,
and it is certainly very slow if $N(|\mathrm{Im}(z)|)$ is small.

We implemented Algorithm~\ref{alg:sketch} with this method for step 5
(first try the fast attempt, and if it fails
use the exhaustive search) and tested it
for quadratic fields of small discriminant.
In practice, this always was fast, taking less than a second to run.
An explanation for this is that
 if $N(|\mathrm{Im}(z)|)$ is small, then the fast LLL-attempt works,
and if the LLL-attempt fails, then $N(|\mathrm{Im}(z)|)$ is large
and hence exhaustive search is fast.

For step 6, write $F=\sum_{i=1}^n f_i X^i Z^{n-i}$ 
and consider the point $p = (\log(f_i))_i \in \RR^{(n+1)d}$.
The goal is to find $v$ in the
lattice $\{(\log(u),\ldots,\log(u))\in \RR^{(n+1)d} : u \in\mathcal{O}_k^*\}$
of rank $d-1$ that is closest to $p$ for the maximum-norm $|\cdot|_{\infty}$.
In the case $d=2$, this lattice has rank~$1$,
and finding a nearest vector in a lattice of rank~$1$ is easy.
Indeed, write $v=k b$ for a basis element $b$ and $k\in \ZZ$
(in our case $b=(\log(\epsilon),\ldots,\log(\epsilon))$ for a fundamental unit $\epsilon$)
and note that the norm $N(k) = |p-kb|_{\infty}$ is convex as a function of~$k$
by the triangle inequality.
By convexity, 
every local minimum is a global minimum,
so we walk from $k=0$ towards a local minimum $k$ and then return $\epsilon^{-k}F$.

We implemented this algorithm in Sage
and made it available online at~\cite{reduce-bitbucket}.

\begin{remark}
If one wants models of hyperelliptic curves
that are isomorphic over~$k$, then simply replace
$\mathcal{O}_k^*$ with $(\mathcal{O}_k^*)^2$ in step~6.
\end{remark}

\section{Results and Tables}
\label{sec:list}

In this section, we give our tables.
The most important columns
(the first and last)
of Tables 1a, 1b and 2b are
explained already in Theorem~\ref{thm:main}.
To explain the rest, we first need
to explain what a CM-type is.

\subsection{CM-types and reflex fields}

A \emph{CM-field} is a totally imaginary quadratic
extension $K$ of a totally real number field~$K_0$.
Note that $K$ has a unique \emph{complex conjugation}
automorphism, which is the generator $\rho =\overline{\cdot}$
of $\mathrm{Gal}(K/K_0)$.
Let $k$ be a field of characteristic zero.
For $\phi:K\rightarrow\kbar$, write
$\overline{\phi}=\phi\circ\rho$.
A \emph{CM-type} of $K$ with values in $\kbar$
is a set $\Phi$ of $g$ embeddings $K\rightarrow\kbar$
such that $\Phi\cup \overline{\Phi}$ is exactly
the set of all $2g$ embeddings.

Let $A$ be an abelian variety 
of dimension $g$ over a field~$k$
of characteristic $0$ and suppose
that $K\cong \End(A_{\kbar})$,
where $K$ is a number field of degree~$2g$.
Choose an isomorphism $i:K\rightarrow \End(A_{\kbar})$
and note that $i$ induces an action of $K$
on the tangent space of $A_{\kbar}$ at zero,
which makes this tangent space
into a $g$-dimensional $\kbar$-linear representation~$R$
of~$K$.
By Complex Multiplication theory
(\cite{shimura-taniyama})
the field $K$ is a CM-field
and there is a CM-type $\Phi$
such that the representation~$R$
is isomorphic
to a direct sum of the $g$ elements of~$\Phi$.
We say that $(A,i)$ is \emph{of type $\Phi$}
and that $\Phi$ is \emph{the CM-type of $(A, i)$}.

The type norm of $\Phi$ is the multiplicative map
\[ N_{\Phi}:K\rightarrow \kbar : \alpha\rightarrow\prod_{\phi\in\Phi} \phi(\alpha),\]
which satisfies $N_{\Phi}(\alpha) = \det R(\alpha)$
if $(A, i)$ is of type~$\Phi$.
The \emph{reflex field}
$K^r\subset \kbar$ is defined to be the field generated
over $\QQ$ by the set of type norms
$\{N_{\Phi}(\alpha)\mid \alpha\in K\}$.
The CM-type and reflex field are important in the theory
of complex multiplication, as they are the link between
the field of definition~$k$ and the endomorphisms in~$K$.
In fact, the main theorem of complex multiplication
involves abelian extension of~$K^r$ rather than~$K$.

Note that the reflex field of the CM-type of $(A,i)$
depends only on~$A$, since composition of $\Phi$
with elements of $\mathrm{Aut}(K)$ does not change $N_{\Phi}$.

\subsection{The case distinctions}\label{ssec:cases}

There are three possibilities for the Galois group
of a quartic CM-field (\cite[Example 8.4(2)]{shimura-taniyama}):
\begin{enumerate}
\item $K/\QQ$ is Galois with cyclic Galois group $C_4$
of order~$4$,
\item $K/\QQ$ is not normal, and its normal closure
has dihedral Galois group $D_4$ of order~$8$,
\item $K/\QQ$ is Galois over $\QQ$ with Galois group $V_4=C_2\times C_2$.
\end{enumerate}
It is known that case 3 of a \emph{biquadratic} CM-field contradicts our assumption
that $A$ is simple over~$\kbar$, so
following the Echidna database~\cite{echidna},
our tables will
be partitioned into cases 1 and~2.

Recall that we are interested in curves with
CM by the maximal order of a quartic CM-field~$K$,
which are defined
over the reflex field~$K^r$.
We distinguish whether the 
curves are defined over:
\begin{itemize}
\item[a.] $\QQ$,
\item[b.] $K^{\mathrm{r}}_0$,  but not $\QQ$,
\item[c.] $K^{\mathrm{r}}$, but not $K^{\mathrm{r}}_0$.
\end{itemize}

The motivation for this article was that case 2a is not
possible, and during our construction of our list
we found no examples for case 1c.
Hence we conjecture that case 1c is empty
and we constructed four tables
corresponding to the four cases 1a, 1b, 2b, and 2c.
Case 1a corresponds to Van Wamelen~\cite{vanwamelen}.

\newcommand{\ZZbar}{\overline{\ZZ}}

\subsection{Legend for the tables}

In case 1, we have $K^r\cong K$ and $\Aut(K)=C_4$,
so every abelian variety with CM by $\mathcal{O}_K$
is of all four CM-types, we therefore give
$K$ and~$f$, but not $\Phi$ or $K^r$.

In case 2, we have two $\Aut(K)$-orbits of CM-types,
and, given~$A$, only one of these orbits correspond to~$A$.
We specify the correct 
CM-type orbit by specifying its reflex field
$K^r$ as an extension of the quadratic field
$K_0^r=\QQ(a)$.

A quartic CM-field $K$ is given up to isomorphism
by a unique triple
$[D,A,B]$ as follows, following the
Echidna database~\cite{echidna}.
Write $K=K_0(\sqrt{r})$ for some real quadratic
field~$K_0$ and some totally negative $r\in K_0$.
Without loss of generality, we take $r\in\mathcal{O}_{K_0}$
with $A=-\tr_{K_0/\QQ}(r)\in\ZZ_{>0}$
\emph{minimal}.
Then let $B=N_{K_0/\QQ}(r)\in\ZZ_{>0}$ and assume $B$ is minimal
for this~$A$.
Finally, let $D=\Delta_{K_0/\QQ}$.
We use the triple
$[D,A,B]$ to represent the isomorphism class
of~$K$, and note $K\cong \QQ[X]/(X^4+AX^2+B)$.

Let us briefly state what the
notation in the table means.

\begin{description}[font=\normalfont,style=sameline,leftmargin=1.3cm]
\item[DAB] With $[D,A,B]$ as in the first column,
let $K=\QQ(\beta)$, where $\beta$ is a root of $X^4+AX^2+B$.
\item[DAB$^r$] In tables 2b and 2c, let $[D^r,A^r,B^r]$
be as in the column DAB$^r$.
Then let $K^r=\QQ(\alpha)$, where $\alpha$ is a root
of $X^4+A^{r}X^2+B^r$. In tables 1a and 1b, we have $K^r\cong K$
and $[D^r,A^r,B^r]=[D,A,B]$.
\item[$a$] A root of
         $X^2+\epsilon X+(D^r-\epsilon)/4$ with $\epsilon\in\{0,1\}$
         congruent to $D^r$ modulo~$4$.
         We have $\ZZ[a]=\mathcal{O}_{K_0^r}$.      
         In case 1, the field $K^r$ is uniquely determined
        as a subset of $\kbar$ by $K^r\cong K$.
         In case 2, there are two quadratic extensions $K^r/\QQ(a)$ that satisfy $K^r\cong \QQ[X]/(X^4+AX^2+B)$, and they are conjugate over $\QQ$. The expression
        of $a$ in terms of $\alpha$ (in the column `$a$')
        tells us which of
        these extensions is $K^r=\QQ(\alpha)$.
\item[$f$, $C$] The polynomial $f\in\ZZ[a][x]$ given in the final
column defines a 
hyperelliptic curve $C:y^2=f(x)$
of genus two.
\item[$\Delta(C)$] The discriminant of the given model
$y^2=f(x)$ of~$C$.
\item[$\Delta_{\mathrm{stable}}$]
The minimal discriminant of all models of $C$
over ${\QQbar}$
of the form $y^2+h(x)y=g(x)$ with
coefficients in~$\ZZbar$.
\item[$\Phi$] One fixed CM-type of $K$ with reflex
    field $K^r$,
    uniquely determined up to right-composition
    with $\mathrm{Aut}(K)$ by the following recipe.
        In case~1, we have $\mathrm{Aut}(K)=C_4$
        and we fix an arbitrary CM-type.
        In case~2, the type $\Phi$ is unique up to complex conjugation and given as follows:
        $\Phi$ is a CM-type of $K$ with values in a normal closure of $K^r$ and reflex field~$K^r$.
\item[$(xa+y)_{n}^{e}$] The $e$th power of the 
       principal $\ZZ[a]$-ideal of norm $n$ generated by $xa+y$.
       This notation is used in the discriminant and obstruction columns.
\end{description}

\subsection{Statement and proof of results regarding the table}

We give the following more detailed version of Theorem~\ref{thm:main}.
\begin{theorem}\label{thm:mainlong}
With the notation as in the legend above,
we have the following.
\begin{enumerate}
\item For every row of Tables 1a, 1b, and~2b,
let $K$ be as specified in that row (see ``\textup{DAB}'' in the legend),
and consider the curves $C$ given in that row.
Then the following holds.
\begin{enumerate}
\item In Table~1a, the given curves
are exactly all $\QQbar$-isomorphism classes
of curves satisfying
$\mathrm{End}(J(C)_{\QQbar})\cong \mathcal{O}_K$.
\item
In Tables 1b and~2b, the given curves and their
quadratic conjugates over~$\QQ$
are exactly all $\QQbar$-isomorphism classes
of curves satisfying
$\mathrm{End}(J(C)_{\QQbar})\cong \mathcal{O}_K$.
\item In Tables 1a and 1b, the curves have CM-type $\Phi$ for every CM-type $\Phi$ of $K$.
\item In Table 2b, the given curve has the given CM-type~$\Phi$,
and its quadratic conjugate has CM-type $\Phi'$ where $\Phi'\not\in\{\Phi,\overline{\Phi}\}$.
\end{enumerate}
\item
The curves in tables 1a, 1b, and~2b
are all defined over $K_0^r$, and
the entries $\Delta(C)/\Delta_{\mathrm{stable}}$
and $\Delta_{\mathrm{stable}}$ are
as explained in the legend above.
\item
In Tables 1b and~2b, the discriminant $\Delta(C)$
is minimal (as defined in Section \ref{ssec:reductiondiscriminantlocal})
 among all $\QQbar$-isomorphic models
of the form $y^2=g(x)$ with $g(x)\in\mathcal{O}_{K^r_0}[x]$,
except for the case of the field $[17, 46, 257]$
in Table~2b,
where a global minimal model does not exist, and
the given model is minimal outside~$(2,a+1)$.
In Table~1a, the discriminant is minimal
among such models with $g(x)\in\ZZ[x]$.
\item 
The curves in Tables 1b and~2b have Igusa invariants
that do not lie in~$\QQ$. In particular, they have
no model over~$\QQ$.
\item
For every row of Table~2c, the number in the final column
is the number of curves over~$\QQbar$ with 
$\mathrm{End}(J(C)_{\QQbar})\cong \mathcal{O}_K$
\emph{of type~$\Phi$}
up to isomorphism over~$\QQbar$.
These curves all have Igusa invariants
in~$K_0^r$ but no model over~$K_0^r$. They do have a model
over~$K^r$. The obstructions column gives exactly the set
of places of~$K_0^r$ at which Mestre's conic locally
has no point.
\end{enumerate}
\end{theorem}

Before we give the proof, let us note that the curves in 1(a) and Table~1a
were already given by Van Wamelen~\cite{vanwamelen}
and proven correct by Van Wamelen~\cite{correctnesswamelen}
and Bisson and Streng~\cite{bisson-streng}.

\begin{proof}
We compute the isomorphism class of the reflex field
as follows.
The reflex field is again a non-biquadratic quartic CM-field.
In fact, one can compute that
it is isomorphic to $\QQ[X]/(X^4+2AX^2+(A^2-4B))$.
Let $[D',A',B']$ be the triple that represents~$K^r$
as before. We do not necessarily have $A'=2A$ and
$B'=A^2-4B$, because those values are not always minimal.
Note that we do
have $K_0^r\cong \QQ(\sqrt{D'})\cong \QQ(\sqrt{B})$.

Our computation of Igusa class polynomials
shows that we have the correct number of curves for each field.
Since we use interval arithmetic and 
the denominator formulas of
Lauter and Viray~\cite{lauter-viray-denominators},
these computations
even prove that the Igusa invariants themselves are correct,
including the ones for Table~2c, which are not listed.
We used the Igusa invariants
to compute the curves and obstructions with Mestre's algorithm,
which proves that the curves and obstructions are correct.
In case~1, all CM-types are in the same orbit for $\mathrm{Aut}(K)$,
so they are all correct. In cases 2b and~2c, 
the correct CM-type is determined using reduction modulo a suitable prime
and the Shimura-Taniyama formula~\cite[Theorem~1(ii) in Section 13.1]{shimura-taniyama}.
Proposition~\ref{prop:local} and our reduction algorithm
prove that the discriminant is minimal.
The stable discriminant is computed directly from Igusa's
arithmetic invariants~\cite{igusa}.
The set of obstructions in Table~2c
is non-empty, hence there is no model over~$K_0^r$.
It remains to prove that there is a model over~$K^r$,
which can be verified by checking that the obstructions
are inert or ramified in $K^r/K_0^r$,
but which also follows from Theorem~\ref{thm:modulidef} below.
\end{proof}

\subsection{Theoretical results}
\label{ssec:theory}

The following known result is the reason
why Van Wamelen's table~\cite{vanwamelen}
did not contain any curves
with CM by non-Galois CM-fields
and why we have no Table 2a.
\begin{proposition}\label{prop:wamelenmissing}
Let $C$ be a curve of genus two with CM by an order
in a non-Galois quartic CM-field.
Then the field of moduli of~$C$ contains~$K_0^r$.
\end{proposition}
\begin{proof}
This is a special case of
\cite[Proposition 5.17(5)]{shimuraAF}.
\end{proof}

While the result above gives a lower bound for the
field of definition and the field of moduli,
the following result
gives an upper bound.

\begin{theorem}\label{thm:modulidef}
Let $C$ be a curve of genus two with CM by 
the maximal order of a non-biquadratic quartic CM-field,
let $K^{\mathrm{r}}$ be the reflex field and $k_0$
the field of moduli.

Then $K^{\mathrm{r}}k_0$ is a field of definition
and we have $[K^{\mathrm{r}}k_0 : K^{\mathrm{r}}_0 k_0] = 2$.
\end{theorem}
\begin{proof}
The first statement is a special case of the main
theorem of Milne~\cite{milne-definition, milne-correction}.
Alternatively, it is Theorem 11 on page 524 of \cite{shimura71},
combined with Proposition 2(3.4) on page 514, with
the line below Proposition 7 on page 525, and
with the fact that there are exactly 2 or 10 roots
of unity in $K$ if $K$ is cyclic or non-Galois of degree 4.

The second statement is a special case of \cite[Lemma 2.6]{streng-reciprocity}.
\end{proof}
\begin{corollary}\label{cor:modulidef}
In the notation of Theorem~\ref{thm:modulidef},
the following are equivalent:
\begin{enumerate}
\item $K^{\mathrm{r}}$ is a field of definition,
\item $K^{\mathrm{r}}$ contains the field of moduli $k_0$,
\item $K^{\mathrm{r}}_0$ contains the field of moduli $k_0$.
\end{enumerate}
In the non-Galois case, these conditions are also equivalent to
\begin{enumerate}\setcounter{enumi}{3}
\item $K^{\mathrm{r}}_0$ equals the field of moduli $k_0$.
\end{enumerate}
\end{corollary}
\begin{proof}
The implications $1\Rightarrow 2$ and $3\Rightarrow 2$ are trivial,
so assume 2 is true. Then Theorem~\ref{thm:modulidef}
states that 1 holds and that $[K^{\mathrm{r}}:K^{\mathrm{r}}_0k_0]=2$ holds,
so 3 also holds.

In the non-Galois case, Proposition~\ref{prop:wamelenmissing}
gives $4\Leftrightarrow 3$.
\end{proof}

\begin{remark}
The main theorem of complex multiplication
gives the Galois group of $k_0K^r/K^r$ as an explicit
quotient of the class group of~$K^r$.
In particular, the conditions of Corollary~\ref{cor:modulidef}
are equivalent to that quotient being trivial.
\end{remark}

The following lemma justifies that we worked under the assumption
$\mathrm{Aut}(C_{\overline{k}})=\{1,\iota\}$ in this paper.

\begin{lemma}\label{lem:aut}

Suppose $C$ is a curve of genus two with CM by an order
$\mathcal{O}\subset K$, and suppose that we are in case
1 or 2 as in Section~\ref{ssec:cases}.
Then either $\mathcal{O}=\ZZ[\zeta_5]$ and $C$
is isomorphic over $\kbar$ to the curve $y^2=x^5-1$ 
(in particular we already know a small model) or
we have $\mathrm{Aut}(C_{\kbar})=\{1, \iota\}$.
\end{lemma}
\begin{proof}
The automorphisms of $C$ correspond to automorphisms
of the principally polarised abelian variety~$J(C)$,
which are roots of unity in $\mathcal{O}=\End(J(C)_{\kbar})$.
The only order in cases 1 and 2 with roots of unity
is $\ZZ[\zeta_5]$, and since it has class number
one, there is only one curve with CM by that ring
up to isomorphism over the algebraic closure.
That curve is the curve $y^2=x^5-1$,
since that has an automorphism of order~$10$.
\end{proof}

\subsection{Completeness}\label{ssec:completeness}

As for completeness, our tables contain all fields
in the Echidna database satisfying $[k_0:\QQ]\leq 2$.
In particular, by Corollary~\ref{cor:modulidef},
our list contains all fields
for which the curve has a model over $K^\mathrm{r}$
as far as the Echidna database has them.
The proof of completeness of this list of fields
is a work in progress of
P{\i}nar K{\i}l{\i}{\c{c}}er.

\vfill\newpage

\newcommand{\curvewidth}{12.5cm}

\addtolength{\voffset}{3cm}
\addtolength{\hoffset}{-1.1in}
\addtolength{\textwidth}{160pt}
\begin{landscape}

\centering
\begin{longtable}{|l|l|l|l|l|l|c|} 
\caption*{Table 1a}\\ \hline
\label{table1astart}

DAB &  $\Delta_{\mathrm{stable}}$ &
$\Delta(C)/\Delta_{\mathrm{stable}}$ & $f$, where $C:y^2=f$  \\ \hline
\endfirsthead
\multicolumn{6}{c}%
{\textit{Table 1a, continued from previous page}}
\\
\hline
DAB & $\Delta_{\mathrm{stable}}$ &
$\Delta(C)/\Delta_{\mathrm{stable}}$ & $f$, where $C:y^2=f$ \\ \hline
\endhead
\hline
\multicolumn{6}{c}{\textit{Continued on next page}}
\\
\endfoot
\hline
\endlastfoot

$ \left[5, 5, 5\right] $ & $ \begin{array}{l} 1 \end{array} $ & $ \begin{array}{l} 2^{8} \cdot 5^{5} \end{array}$ & \parbox{\curvewidth}{$ x^{5} - 1 $}     \\  \hline
\multirow{2}{*}{\raggedleft $ \left[5, 10, 20\right] $ } & $ \begin{array}{l} 2^{12} \end{array} $ & $ \begin{array}{l} 2^{10} \cdot 5^{5} \end{array}$ & \parbox{\curvewidth}{$ 4 x^{5} - 30 x^{3} + 45 x - 22 $}     \\  \cline{2-5}
 & $ \begin{array}{l} 2^{12} \cdot 11^{12} \end{array} $ & $ \begin{array}{l} 2^{10} \cdot 5^{5} \end{array}$ & \parbox{\curvewidth}{$ 8 x^{6} + 52 x^{5} - 250 x^{3} + 321 x - 131 $}     \\  \hline
\multirow{2}{*}{\raggedleft  $ \left[5, 65, 845\right] $}  & $ \begin{array}{l} 11^{12} \end{array} $ & $ \begin{array}{l} 2^{20} \cdot 5^{5} \cdot 13^{10} \end{array}$ & \parbox{\curvewidth}{$ 8 x^{6} - 112 x^{5} - 680 x^{4} + 8440 x^{3} + 28160 x^{2} - 55781 x + 111804 $}     \\  \cline{2-5}
 & $ \begin{array}{l} 31^{12} \cdot 41^{12} \end{array} $ & $ \begin{array}{l} 2^{20} \cdot 5^{5} \cdot 13^{10} \end{array}$ & \parbox{\curvewidth}{$ -9986 x^{6} + 73293 x^{5} - 348400 x^{3} - 118976 x - 826072 $}     \\  \hline
\multirow{2}{*}{\raggedleft  $ \left[5, 85, 1445\right] $} & $ \begin{array}{l} 71^{12} \end{array} $ & $ \begin{array}{l} 2^{20} \cdot 5^{5} \cdot 17^{10} \end{array}$ & \parbox{\curvewidth}{$ -73 x^{6} + 1005 x^{5} + 14430 x^{4} - 130240 x^{3} - 1029840 x^{2} + 760976 x - 2315640 $}     \\  \cline{2-5}
 & $ \begin{array}{l} 11^{12} \cdot 41^{12} \cdot 61^{12} \end{array} $ & $ \begin{array}{l} 2^{20} \cdot 5^{5} \cdot 17^{10} \end{array}$ & \parbox{\curvewidth}{$ 2160600 x^{6} - 8866880 x^{5} + 2656360 x^{4} - 582800 x^{3} + 44310170 x^{2} + 6986711 x - 444408 $}     \\  \hline
$ \left[8, 4, 2\right] $ & $ \begin{array}{l} 2^{6} \end{array} $ & $ \begin{array}{l} 2^{15} \end{array}$ & \parbox{\curvewidth}{$ x^{5} - 3 x^{4} - 2 x^{3} + 6 x^{2} + 3 x - 1 $}     \\  \hline
\multirow{2}{*}{\raggedleft  $ \left[8, 20, 50\right] $} & $ \begin{array}{l} 2^{6} \cdot 7^{12} \cdot 23^{12} \end{array} $ & $ \begin{array}{l} 2^{15} \cdot 5^{10} \end{array}$ & \parbox{\curvewidth}{$ -8 x^{6} - 530 x^{5} + 160 x^{4} + 64300 x^{3} - 265420 x^{2} - 529 x $}     \\  \cline{2-5}
 & $ \begin{array}{l} 2^{6} \cdot 7^{12} \cdot 17^{12} \cdot 23^{12} \end{array} $ & $ \begin{array}{l} 2^{15} \cdot 5^{10} \end{array}$ & \parbox{\curvewidth}{$ 4116 x^{6} + 64582 x^{5} + 139790 x^{4} - 923200 x^{3} + 490750 x^{2} + 233309 x - 9347 $}     \\  \hline
$ \left[13, 13, 13\right] $ & $ \begin{array}{l} 1 \end{array} $ & $ \begin{array}{l} 2^{20} \cdot 13^{5} \end{array}$ & \parbox{\curvewidth}{$ x^{6} - 8 x^{4} - 8 x^{3} + 8 x^{2} + 12 x - 8 $}     \\  \hline
\multirow{2}{*}{\raggedleft  $ \left[13, 26, 52\right] $} & $ \begin{array}{l} 2^{12} \cdot 3^{12} \cdot 23^{12} \end{array} $ & $ \begin{array}{l} 2^{10} \cdot 13^{5} \end{array}$ & \parbox{\curvewidth}{$ -243 x^{6} - 2223 x^{5} - 1566 x^{4} + 19012 x^{3} + 903 x^{2} - 19041 x - 5882 $}     \\  \cline{2-5}
 & $ \begin{array}{l} 2^{12} \cdot 3^{12} \cdot 23^{12} \cdot 131^{12} \end{array} $ & $ \begin{array}{l} 2^{10} \cdot 13^{5} \end{array}$ & \parbox{\curvewidth}{$ 59499 x^{6} - 125705 x^{5} - 801098 x^{4} + 1067988 x^{3} + 2452361 x^{2} + 707297 x - 145830 $}     \\  \hline
\multirow{2}{*}{\raggedleft  $ \left[13, 65, 325\right] $} & $ \begin{array}{l} 3^{12} \end{array} $ & $ \begin{array}{l} 2^{20} \cdot 5^{10} \cdot 13^{5} \end{array}$ & \parbox{\curvewidth}{$ 36 x^{5} - 1040 x^{3} + 1560 x^{2} + 1560 x + 1183 $}     \\  \cline{2-5}
 & $ \begin{array}{l} 3^{12} \cdot 53^{12} \end{array} $ & $ \begin{array}{l} 2^{20} \cdot 5^{10} \cdot 13^{5} \end{array}$ & \parbox{\curvewidth}{$ -1323 x^{6} - 1161 x^{5} + 9360 x^{4} + 9590 x^{3} - 34755 x^{2} + 1091 x + 32182 $}     \\  \hline
$ \left[29, 29, 29\right] $ & $ \begin{array}{l} 5^{12} \end{array} $ & $ \begin{array}{l} 2^{20} \cdot 29^{5} \end{array}$ & \parbox{\curvewidth}{$ 43 x^{6} - 216 x^{5} + 348 x^{4} - 348 x^{2} - 116 x $}     \\  \hline
$ \left[37, 37, 333\right] $ & $ \begin{array}{l} 3^{12} \cdot 11^{12} \end{array} $ & $ \begin{array}{l} 2^{20} \cdot 37^{5} \end{array}$ & \parbox{\curvewidth}{$ -68 x^{6} + 57 x^{5} + 84 x^{4} - 680 x^{3} + 72 x^{2} - 1584 x - 4536 $}     \\  \hline
$ \left[53, 53, 53\right] $ & $ \begin{array}{l} 17^{12} \cdot 29^{12} \end{array} $ & $ \begin{array}{l} 2^{20} \cdot 53^{5} \end{array}$ & \parbox{\curvewidth}{$ -3800 x^{6} + 15337 x^{5} + 160303 x^{4} - 875462 x^{3} + 896582 x^{2} - 355411 x + 50091 $}     \\  \hline
$ \left[61, 61, 549\right] $ & $ \begin{array}{l} 3^{24} \cdot 5^{12} \cdot 41^{12} \end{array} $ & $ \begin{array}{l} 2^{20} \cdot 61^{5} \end{array}$ & \parbox{\curvewidth}{$ 40824 x^{6} + 103680 x^{5} - 67608 x^{4} - 197944 x^{3} - 17574 x^{2} + 41271 x + 103615 $}     \\  \hline

\end{longtable}

\renewcommand{\curvewidth}{10.5cm}

\centering
\begin{longtable}{|l|l|l|l|l|l|c|} 
\caption*{Table 1b}\\ \hline

DAB &  $\Delta_{\mathrm{stable}}$ &
$\Delta(C)/\Delta_{\mathrm{stable}}$ & $f$, where $C:y^2=f$ \\ \hline
\endfirsthead
\multicolumn{6}{c}%
{\textit{Table 1b, continued from previous page}}
\\
\hline
DAB & $\Delta_{\mathrm{stable}}$ &
$\Delta(C)/\Delta_{\mathrm{stable}}$ & $f$, where $C:y^2=f$ \\ \hline
\endhead
\hline
\multicolumn{6}{c}{\textit{Continued on next page}}
\\
\endfoot
\hline
\endlastfoot
\multirow{2}{*}{ $\raggedleft \left[5, 15, 45\right] $ }  & $ \begin{array}{l} (2)^{12} \cdot (3)^{6} \end{array} $ & $ \begin{array}{l} (2 a + 1)_{5}^{10} \end{array}$ & \parbox{\curvewidth}{$ -x^{6} + \left(-3 a - 3\right) x^{5} + \left(5 a + 15\right) x^{3} + \left(-15 a - 3\right) x - 4 a + 1 $}     \\  \cline{2-5}
 & $ \begin{array}{l} (2)^{12} \cdot (3)^{6} \cdot (5 a + 2)_{31}^{12} \end{array} $ & $ \begin{array}{l} (2 a + 1)_{5}^{10} \end{array}$ & \parbox{\curvewidth}{$ \left(-2 a + 3\right) x^{6} + \left(-9 a + 18\right) x^{5} + \left(15 a - 70\right) x^{3} + \left(39 a + 54\right) x - 52 a - 1 $}     \\  \hline
\multirow{2}{*}{ $\raggedleft \left[5, 30, 180\right] $} & $ \begin{array}{l} (3 a + 2)_{11}^{12} \cdot (2)^{18} \cdot (3)^{6} \cdot (5 a + 2)_{31}^{12} \end{array} $ & $ \begin{array}{l} (2 a + 1)_{5}^{10} \end{array}$ & \parbox{\curvewidth}{$ 684 x^{6} + \left(390 a + 90\right) x^{5} + \left(24 a - 3138\right) x^{4} + \left(217 a + 401\right) x^{3} +\\ \left(96 a + 3918\right) x^{2} + \left(-2112 a - 1698\right) x + 284 a + 432 $}     \\  \cline{2-5}
 & $ \begin{array}{l} (3 a + 1)_{11}^{12} \cdot (2 a - 11)_{139}^{12} \cdot (4 a + 3)_{19}^{12} \\ \cdot (2)^{18}  \cdot (3)^{6} \cdot (5 a + 2)_{31}^{12} \end{array} $ & $ \begin{array}{l} (2 a + 1)_{5}^{10} \end{array}$ & \parbox{\curvewidth}{$ \left(927 a + 2906\right) x^{6} + \left(5541 a + 18822\right) x^{5} + \left(-33535 a - 124380\right) x^{3} +\\ \left(33417 a + 183726\right) x + 12641 a - 31928 $}     \\  \hline
\multirow{2}{*}{ $\raggedleft \left[5, 35, 245\right] $} & $ \begin{array}{l} (3 a + 2)_{11}^{12} \cdot (2)^{12} \cdot (a + 6)_{29}^{12}\\ \cdot (7)^{6} \cdot (a + 9)_{71}^{12} \end{array} $ & $ \begin{array}{l} (2 a + 1)_{5}^{10} \end{array}$ & \parbox{\curvewidth}{$ \left(-4527 a - 783\right) x^{6} + \left(6392 a + 7811\right) x^{5} + \left(-4500 a - 17085\right) x^{3} + \\ \left(-6948 a + 9783\right) x - 1687 a + 39 $}     \\  \cline{2-5}
 & $  \begin{array}{l} (3 a + 1)_{11}^{12} \cdot (11 a + 5)_{151}^{12} \\ \cdot (2 a + 15)_{191}^{12} \cdot (2)^{12} \\ \cdot (a - 5)_{29}^{12} \cdot (7)^{6} \end{array} $ & $ \begin{array}{l} (2 a + 1)_{5}^{10} \end{array}$ & \parbox{\curvewidth}{$ \left(-435 a - 521\right) x^{6} + \left(353 a + 110\right) x^{5} + \left(131927 a + 189531\right) x^{4} + \\  \left(-696187 a - 952511\right) x^{3} + \left(-10094248 a - 15393369\right) x^{2} + \\ \left(94869598 a + 145990333\right) x - 210533420 a - 329328479 $}      \\ \hline
\multirow{2}{*}{$\raggedleft \left[5, 105, 2205\right] $} & $ \begin{array}{l} (3 a + 1)_{11}^{12}  \cdot (3)^{6} \cdot (7)^{6} \end{array} $ & $ \begin{array}{l} (2)^{20} \cdot (2 a + 1)_{5}^{10} \end{array}$ & \parbox{\curvewidth}{$ \left(-5 a + 4\right) x^{6} + \left(-81 a + 30\right) x^{5} + \left(-135 a + 210\right) x^{4} + \left(450 a - 210\right) x^{3} +\\ \left(360 a - 1785\right) x^{2} + \left(600 a + 15\right) x - 950 a + 5625 $}     \\  \cline{2-5}
 & $ \begin{array}{l} (a + 11)_{109}^{12} \cdot (3 a + 2)_{11}^{12} \cdot (3)^{6} \\ \cdot (7)^{6} \cdot (8 a + 3)_{79}^{12} \end{array} $ & $ \begin{array}{l} (2)^{20} \cdot (2 a + 1)_{5}^{10} \end{array}$ & \parbox{\curvewidth}{$ \left(-3 a - 260\right) x^{6} + \left(1032 a + 1389\right) x^{5} + \left(19160 a + 8760\right) x^{3} +\\ \left(-16224 a + 163200\right) x + 162976 a + 114632 $}     \\  \hline
$ \left[8, 12, 18\right] $ & $ \begin{array}{l} (a)_{2}^{12} \cdot (3)^{6} \\ \cdot (2 a - 1)_{7}^{12} \cdot (2 a + 1)_{7}^{12} \end{array} $ & $ \begin{array}{l} (a)_{2}^{30} \end{array}$ & \parbox{\curvewidth}{$ \left(24 a - 54\right) x^{5} + \left(-66 a + 96\right) x^{4} + \left(-32 a + 220\right) x^{3} + \left(12 a - 312\right) x^{2} +\\ \left(96 a + 21\right) x - 5 a - 16 $}     \\  \hline
$ \left[17, 119, 3332\right] $ & $ \begin{array}{l} (2 a + 15)_{179}^{12} \cdot (a + 2)_{2}^{36} \\ \cdot (a - 1)_{2}^{12} \cdot (4 a + 7)_{43}^{12} \cdot (7)^{6} \end{array} $ & $ \begin{array}{l} (2 a + 1)_{17}^{10} \end{array}$ & \parbox{\curvewidth}{$ \left(213 a + 1875\right) x^{6} + \left(8071 a + 4059\right) x^{5} + \left(-1045 a + 58039\right) x^{4} + \\ \left(32898 a + 26657\right) x^{3} + \left(-12585 a + 3550\right) x^{2} + \left(-46889 a - 136176\right) x \\- 42057 a - 104692 $}     \\  \hline
\multirow{2}{*}{$\raggedleft \left[17, 255, 15300\right] $} & $ \begin{array}{l} (2 a - 5)_{19}^{12} \cdot (a + 2)_{2}^{24} \\ \cdot (a - 1)_{2}^{24} \cdot (3)^{6} \cdot (2 a + 31)_{883}^{12} \end{array} $ & $ \begin{array}{l} (2 a + 1)_{17}^{10} \cdot (5)^{10} \end{array}$ & \parbox{\curvewidth}{$ \left(-4264 a - 13208\right) x^{6} + \left(9516 a - 94116\right) x^{5} + \left(331770 a - 503670\right) x^{4} +\\ \left(-1195640 a + 1593625\right) x^{3} + \left(1141785 a - 2476410\right) x^{2} +\\ \left(-69927 a + 2540472\right) x - 301251 a - 1280828 $}     \\  \cline{2-5}
 & $ \begin{array}{l} (2 a + 3)_{13}^{12} \cdot (4 a + 17)_{157}^{12}  \cdot (2 a + 7)_{19}^{12}\\ \cdot (a + 2)_{2}^{12}  \cdot (a - 1)_{2}^{12} \cdot (3)^{6} \\ \cdot (4 a + 3)_{67}^{12} \cdot (2 a - 9)_{83}^{12}  \cdot (2 a + 11)_{83}^{12} \end{array} $ & $ \begin{array}{l} (2 a + 1)_{17}^{10} \cdot (5)^{10} \end{array}$ & \parbox{\curvewidth}{$ \left(3703196 a + 9037010\right) x^{6} + \\ \left(12666396 a + 36366348\right) x^{5} + \left(33133830 a + 56148570\right) x^{4} +\\ \left(35333760 a + 111063545\right) x^{3} + \left(71845845 a + 45282705\right) x^{2} + \\ \left(154100103 a - 105860229\right) x + 81081415 a - 36366223 $}     \\  \hline

\end{longtable}

\renewcommand{\curvewidth}{9.5cm}

\centering
\begin{longtable}{|l|l|l|l|l|l|c|} 
\caption*{Table 2b}\\ \hline

DAB & DAB$^r$ & $  a $ & $\Delta_{\mathrm{stable}}$ &
$\Delta(C)/\Delta_{\mathrm{stable}}$ & $f$, where
$C : y^2 = f$ \\ \hline
\endfirsthead
\multicolumn{6}{c}%
{\textit{Table 2b, continued from previous page}}
\\
\hline
DAB & DAB$^r$ & $  a $ & $\Delta_{\mathrm{stable}}$ &
$\Delta(C)/\Delta_{\mathrm{stable}}$ & $f$, where
$C : y^2 = f$ \\ \hline
\endhead
\hline
\multicolumn{6}{c}{\textit{Continued on next page}}
\\
\endfoot
\hline
\endlastfoot
\multirow{2}{*}{$ \left[5, 11, 29\right] $} & \multirow{2}{*}{$ \left[29, 7, 5\right] $} & \multirow{2}{*}{$ \alpha^{2} + 3 $} & $ \begin{array}{l} (2)^{12} \cdot (a - 1)_{5}^{12} \cdot (a + 1)_{7}^{12} \end{array} $ & $ \begin{array}{l} (a + 2)_{5}^{10} \end{array}$ & \parbox{\curvewidth}{$ \left(18 a + 60\right) x^{6} + \left(-76 a - 246\right) x^{5} + \left(127 a + 329\right) x^{4} + \\ \left(-77 a - 209\right) x^{3} + \left(-30 a + 155\right) x^{2} + \\ \left(29 a - 69\right) x + 71 a - 156 $}     \\  \cline{4-7}
& & & $ \begin{array}{l} (2)^{12} \cdot (a + 6)_{23}^{12} \cdot (a - 1)_{5}^{12} \end{array} $ & $ \begin{array}{l} (a + 2)_{5}^{10} \end{array}$ & \parbox{\curvewidth}{$ \left(2 a + 1\right) x^{6} + \left(-a - 26\right) x^{5} + \left(9 a + 38\right) x^{4} +  \left(-40 a - 25\right) x^{3} + \\ \left(-21 a - 37\right) x^{2} +  \left(100 a + 218\right) x + 102 a + 268 $}     \\  \hline
$ \left[5, 13, 41\right] $ & $ \left[41, 11, 20\right] $ & $ \alpha^{2} + 5 $ & $ \begin{array}{l} (a - 3)_{2}^{12} \end{array} $ & $ \begin{array}{l} (a + 4)_{2}^{20} \cdot (2 a - 5)_{5}^{10} \end{array}$ & \parbox{\curvewidth}{$ \left(-a + 3\right) x^{6} + \left(4 a - 8\right) x^{5} + 10 x^{4} + \left(-a + 20\right) x^{3} + \left(4 a + 5\right) x^{2} + \left(a + 4\right) x + 1 $}    \\  \hline
$ \left[5, 17, 61\right] $ & $ \left[61, 9, 5\right] $ & $ \alpha^{2} + 4 $ & $ \begin{array}{l} (a - 3)_{3}^{12} \end{array} $ & $ \begin{array}{l} (2)^{20} \cdot (a - 4)_{5}^{10} \end{array}$ & \parbox{\curvewidth}{$ \left(a + 4\right) x^{6} + \left(-8 a - 42\right) x^{5} + \left(37 a + 117\right) x^{4} +  \left(-20 a - 240\right) x^{3} + \\ \left(56 a - 9\right) x^{2} +  \left(22 a - 114\right) x + 9 a - 28 $}    \\  \hline
$ \left[5, 21, 109\right] $ & $ \left[109, 17, 45\right] $ & $ \alpha^{2} + 8 $ & $ \begin{array}{l} (a - 5)_{3}^{12} \cdot (3 a + 17)_{5}^{12} \end{array} $ & $ \begin{array}{l} (2)^{20} \cdot (3 a - 14)_{5}^{10} \end{array}$ & \parbox{\curvewidth}{$ \left(-28 a + 53\right) x^{6} + \left(-113 a + 913\right) x^{5} + \left(-495 a + 1890\right) x^{4} + \\ \left(-746 a + 3308\right) x^{3} + \left(-563 a + 3574\right) x^{2} + \\ \left(-378 a + 1069\right) x - 151 a - 227 $}    \\  \hline
$ \left[5, 26, 149\right] $ & $ \left[149, 13, 5\right] $ & $ \alpha^{2} + 6 $ & $ \begin{array}{l} (a + 7)_{5}^{12} \cdot (a - 5)_{7}^{12} \end{array} $ & $ \begin{array}{l} (2)^{20} \cdot (a - 6)_{5}^{10} \end{array}$ & \parbox{\curvewidth}{$ \left(-125 a - 875\right) x^{6} \\ + \left(-1375 a - 8575\right) x^{5} + \left(-9090 a - 62160\right) x^{4} + \\ \left(-38862 a - 251798\right) x^{3} + \left(-73257 a - 489843\right) x^{2} + \\ \left(-53235 a - 347403\right) x - 12896 a - 86314 $}    \\  \hline
\multirow{2}{*}{$ \left[5, 33, 261\right] $} & \multirow{2}{*}{$ \left[29, 21, 45\right] $} & \multirow{2}{*}{$ \frac{1}{3} \alpha^{2} + 3 $} & $ \begin{array}{l} (a + 5)_{13}^{12} \cdot (3)^{6} \end{array} $ & $ \begin{array}{l} (2)^{20} \cdot (a + 2)_{5}^{10} \end{array}$ & \parbox{\curvewidth}{$ \left(-27 a - 96\right) x^{5} + \left(-18 a - 51\right) x^{4} + \\ \left(-34 a - 58\right) x^{3} + \left(-18 a - 36\right) x^{2} - 15 x - 9 a - 27 $}     \\  \cline{4-7}
& & & $ \begin{array}{l} (3)^{6} \cdot (a)_{7}^{12} \end{array} $ & $ \begin{array}{l} (2)^{20} \cdot (a + 2)_{5}^{10} \end{array}$ & \parbox{\curvewidth}{$ \left(-3 a + 6\right) x^{5} - 90 x^{4} + \left(-128 a - 136\right) x^{3} + \left(-72 a - 744\right) x^{2} + \\ \left(-240 a - 240\right) x - 216 $}     \\  \hline
$ \left[5, 34, 269\right] $ & $ \left[269, 17, 5\right] $ & $ \alpha^{2} + 8 $ & $ \begin{array}{l} (a - 7)_{11}^{12} \cdot (2 a - 15)_{13}^{12} \\ \cdot (a + 9)_{5}^{12} \end{array} $ & $ \begin{array}{l} (2)^{20} \cdot (a - 8)_{5}^{10} \end{array}$ & \parbox{\curvewidth}{$ \left(-283 a + 2246\right) x^{6} + \\ \left(-4563 a + 33800\right) x^{5} + \left(-11932 a + 103166\right) x^{4} + \\ \left(127408 a - 1032304\right) x^{3} + \left(998576 a - 7558008\right) x^{2} + \\ \left(2439792 a - 18969664\right) x + 2110776 a - 16149072 $}    \\  \hline
$ \left[5, 41, 389\right] $ & $ \left[389, 37, 245\right] $ & $ \alpha^{2} + 18 $ & $ \begin{array}{l} (2 a + 21)_{11}^{12} \cdot (8 a + 83)_{17}^{12} \\ \cdot (5 a + 52)_{19}^{12} \cdot (3 a - 28)_{5}^{12} \end{array} $ & $ \begin{array}{l} (2)^{20} \cdot (3 a + 31)_{5}^{10} \end{array}$ & \parbox{\curvewidth}{$ \left(1248 a - 11685\right) x^{6} + \\ \left(-16097 a + 150611\right) x^{5} + \left(37185 a - 349530\right) x^{4} + \\ \left(250806 a - 2359968\right) x^{3} + \left(-972081 a + 9046728\right) x^{2} + \\ \left(-942318 a + 8701533\right) x + 4994791 a - 46866753 $}    \\  \hline
\multirow{2}{*}{$ \left[5, 66, 909\right] $} & \multirow{2}{*}{$ \left[101, 33, 45\right] $} & \multirow{2}{*}{$ \frac{1}{3} \alpha^{2} + 5 $} & $ \begin{array}{l} (a - 2)_{19}^{12} \cdot (3)^{6} \\ \cdot (2 a + 13)_{43}^{12} \cdot (a - 4)_{5}^{12} \end{array} $ & $ \begin{array}{l} (2)^{20} \cdot (a + 5)_{5}^{10} \end{array}$ & \parbox{\curvewidth}{$ \left(-340 a - 1674\right) x^{6} + \\ \left(-4179 a - 26820\right) x^{5} + \left(-26433 a - 118800\right) x^{4} + \\ \left(-38358 a - 315240\right) x^{3} + \left(-46686 a - 41130\right) x^{2} + \\ \left(40761 a - 15348\right) x - 13013 a + 39100 $}     \\  \cline{4-7}
& & & $ \begin{array}{l} (3)^{6} \cdot (a + 8)_{31}^{12}\\ \cdot (2 a - 7)_{37}^{12} \cdot (a - 4)_{5}^{12} \end{array} $ & $ \begin{array}{l} (2)^{20} \cdot (a + 5)_{5}^{10} \end{array}$ & \parbox{\curvewidth}{$ \left(-6120 a - 36189\right) x^{6}  +  \left(-22143 a - 102375\right) x^{5} + \\ \left(-21378 a - 184140\right) x^{4} +  \left(-31356 a - 65810\right) x^{3} +\\ \left(765 a - 81765\right) x^{2} + \\ \left(-3783 a + 6192\right) x $}     \\  \hline
$ \left[8, 10, 17\right] $ & $ \left[17, 5, 2\right] $ & $ \alpha^{2} + 2 $ & $ \begin{array}{l} (a + 2)_{2}^{6} \end{array} $ & $ \begin{array}{l} (a + 2)_{2}^{45} \cdot (a - 1)_{2}^{20} \end{array}$ & \parbox{\curvewidth}{$ x^{6} + \left(2 a + 4\right) x^{5} + \left(3 a + 14\right) x^{4} +  \left(10 a + 8\right) x^{3} +\\ \left(-9 a + 32\right) x^{2} +  \left(16 a - 16\right) x - 4 a + 8 $}     \\  \hline
$ \left[8, 18, 73\right] $ & $ \left[73, 9, 2\right] $ & $ \alpha^{2} + 4 $ & $ \begin{array}{l} (a - 4)_{2}^{6} \cdot (a + 5)_{2}^{12} \\ \cdot (4 a - 15)_{3}^{12} \end{array} $ & $ \begin{array}{l} (a - 4)_{2}^{45} \end{array}$ & \parbox{\curvewidth}{$ \left(a + 5\right) x^{6} + \left(28 a + 132\right) x^{5} + \left(214 a + 1026\right) x^{4} + \\ \left(349 a + 1658\right) x^{3} + \left(259 a + 1242\right) x^{2} + \\ \left(47 a + 222\right) x - 3 a - 14 $}     \\  \hline
$ \left[8, 22, 89\right] $ & $ \left[89, 11, 8\right] $ & $ \alpha^{2} + 5 $ & $ \begin{array}{l} (a - 4)_{2}^{12} \cdot (a + 5)_{2}^{6} \\ \cdot (4 a - 17)_{5}^{12} \end{array} $ & $ \begin{array}{l} (a + 5)_{2}^{45} \end{array}$ & \parbox{\curvewidth}{$ \left(a - 4\right) x^{6} + \left(8 a - 36\right) x^{5} + \left(16 a - 62\right) x^{4} +  \left(-13 a + 57\right) x^{3} + \\ \left(-17 a + 73\right) x^{2} +  \left(13 a - 57\right) x - a + 5 $}     \\  \hline
$ \left[8, 34, 281\right] $ & $ \left[281, 17, 2\right] $ & $ \alpha^{2} + 8 $ & $ \begin{array}{l} (42 a - 331)_{17}^{12} \cdot (a - 8)_{2}^{6} \\ \cdot (a + 9)_{2}^{24} \cdot (76 a + 675)_{5}^{12}\\ \cdot (8 a - 63)_{7}^{12} \end{array} $ & $ \begin{array}{l} (a - 8)_{2}^{45} \end{array}$ & \parbox{\curvewidth}{$ \left(-15024 a + 118185\right) x^{6} + \\ \left(310153 a - 2435026\right) x^{5} + \left(-2658057 a + 20990488\right) x^{4} + \\ \left(12047831 a - 97400942\right) x^{3} + \left(-33280854 a + 231380920\right) x^{2} + \\ \left(34989188 a - 413796872\right) x - 37610304 a + 81055944 $}     \\  \hline
$ \left[8, 38, 233\right] $ & $ \left[233, 19, 32\right] $ & $ \alpha^{2} + 9 $ & $ \begin{array}{l} (38 a - 271)_{13}^{12} \cdot (a + 8)_{2}^{12} \\ \cdot (a - 7)_{2}^{6} \cdot (8 a + 65)_{7}^{12}\\ \cdot (8 a - 57)_{7}^{12} \end{array} $ & $ \begin{array}{l} (a - 7)_{2}^{45} \end{array}$ & \parbox{\curvewidth}{$ \left(-166628 a - 1355047\right) x^{6} + \\ \left(-354121 a - 2879769\right) x^{5} + \left(-318274 a - 2588269\right) x^{4} + \\ \left(-153661 a - 1249743\right) x^{3} + \left(-41827 a - 339754\right) x^{2} + \\ \left(-6158 a - 48444\right) x - 441 a - 2400 $}     \\  \hline
\multirow{2}{*}{$ \left[8, 50, 425\right] $} & \multirow{2}{*}{$ \left[17, 25, 50\right] $} & \multirow{2}{*}{$ \frac{1}{5} \alpha^{2} + 2 $} & $ \begin{array}{l} (a + 2)_{2}^{6} \cdot (a - 1)_{2}^{12} \cdot (5)^{6} \end{array} $ & $ \begin{array}{l} (a + 2)_{2}^{45} \cdot (5)^{15} \end{array}$ & \parbox{\curvewidth}{$ \left(34 a + 80\right) x^{6} + \left(140 a + 224\right) x^{5} + \left(110 a - 220\right) x^{4} + \\ \left(-455 a + 220\right) x^{3} + \left(-5 a + 190\right) x^{2} + \\ \left(91 a - 104\right) x + 254 a - 395 $}     \\  \cline{4-7}
& & & $ \begin{array}{l} (2 a + 3)_{13}^{12} \cdot (2 a - 5)_{19}^{12} \\ \cdot (a + 2)_{2}^{6} \cdot (a - 1)_{2}^{24} \cdot (5)^{6} \end{array} $ & $ \begin{array}{l} (a + 2)_{2}^{45} \cdot (5)^{15} \end{array}$ & \parbox{\curvewidth}{$ \left(-1455 a + 1511\right) x^{6} + \\ \left(-1004 a - 2656\right) x^{5} + \left(-19100 a + 20290\right) x^{4} + \\ \left(-3805 a - 4380\right) x^{3} + \left(-72745 a + 108600\right) x^{2} + \\ \left(-7451 a + 10748\right) x - 99295 a + 155108 $}     \\  \hline
\multirow{2}{*}{$ \left[8, 66, 1017\right] $} & \multirow{2}{*}{$ \left[113, 33, 18\right] $} & \multirow{2}{*}{$ \frac{1}{3} \alpha^{2} + 5 $} & $ \begin{array}{l} (4 a - 19)_{11}^{12} \cdot (a + 6)_{2}^{12} \\ \cdot (a - 5)_{2}^{6} \cdot (3)^{6}\\  \cdot (8 a + 47)_{41}^{12} \cdot (6 a + 35)_{7}^{12}\\ \cdot (6 a - 29)_{7}^{12} \end{array} $ & $ \begin{array}{l} (a - 5)_{2}^{45} \end{array}$ & \parbox{\curvewidth}{$ \left(-4215 a - 14698\right) x^{6} + \\ \left(30036 a + 338652\right) x^{5} + \left(-549576 a - 134610\right) x^{4} + \\ \left(-2945519 a + 22716733\right) x^{3} + \left(12849441 a - 76601511\right) x^{2} + \\ \left(234523575 a - 1115687637\right) x - 843111919 a + 4054444133 $}     \\  \cline{4-7}
& & & $ \begin{array}{l} (a + 6)_{2}^{12} \cdot (a - 5)_{2}^{6}\\ \cdot (3)^{6} \cdot (2 a + 13)_{31}^{12}\\ \cdot (28 a + 163)_{53}^{12} \cdot (6 a + 35)_{7}^{12} \end{array} $ & $ \begin{array}{l} (a - 5)_{2}^{45} \end{array}$ & \parbox{\curvewidth}{$ \left(-27 a - 2538\right) x^{6} + \\ \left(7230 a + 8412\right) x^{5} + \left(-3867 a - 272622\right) x^{4} + \\ \left(121693 a + 458725\right) x^{3} + \left(-1686144 a + 6014715\right) x^{2} + \\ \left(-5324007 a + 27892107\right) x + 110392412 a - 532554277 $}     \\  \hline
$ \left[13, 9, 17\right] $ & $ \left[17, 15, 52\right] $ & $ \alpha^{2} + 7 $ & $ \begin{array}{l} (a + 2)_{2}^{12} \end{array} $ & $ \begin{array}{l} (2 a - 1)_{13}^{10} \cdot (a - 1)_{2}^{20} \end{array}$ & \parbox{\curvewidth}{$ \left(a - 2\right) x^{6} + \left(-8 a + 8\right) x^{5} + \left(14 a - 32\right) x^{4} +  \left(-19 a + 27\right) x^{3} + \\ \left(6 a - 21\right) x^{2} +  \left(3 a + 9\right) x - 4 a - 7 $}    \\  \hline
$ \left[13, 18, 29\right] $ & $ \left[29, 9, 13\right] $ & $ \alpha^{2} + 4 $ & $ \begin{array}{l} (a - 1)_{5}^{12} \end{array} $ & $ \begin{array}{l} (a - 4)_{13}^{10} \cdot (2)^{20} \end{array}$ & \parbox{\curvewidth}{$ \left(9 a - 22\right) x^{6} + \left(-19 a + 21\right) x^{5} + \left(8 a - 95\right) x^{4} + \\ \left(-70 a - 6\right) x^{3} + \left(-23 a - 148\right) x^{2} + \\ \left(-7 a - 127\right) x - 18 a - 7 $}    \\  \hline
$ \left[13, 29, 181\right] $ & $ \left[181, 41, 13\right] $ & $ \frac{1}{3} \alpha^{2} + \frac{19}{3} $ & $ \begin{array}{l} (6 a - 37)_{29}^{12} \cdot (a - 6)_{3}^{12} \\ \cdot (a + 7)_{3}^{12} \cdot (4 a + 29)_{5}^{12} \end{array} $ & $ \begin{array}{l} (3 a - 19)_{13}^{10} \cdot (2)^{20} \end{array}$ & \parbox{\curvewidth}{$ \left(-16581 a - 119826\right) x^{6} + \\ \left(-52472 a - 379062\right) x^{5} + \left(-67729 a - 508419\right) x^{4} + \\ \left(-78876 a - 162464\right) x^{3} + \left(-44960 a + 21657\right) x^{2} + \\ \left(14402 a - 144114\right) x - 21885 a + 131494 $}    \\  \hline
$ \left[13, 41, 157\right] $ & $ \left[157, 25, 117\right] $ & $ \alpha^{2} + 12 $ & $ \begin{array}{l} (3 a + 20)_{11}^{12} \cdot (a - 7)_{17}^{12} \\ \cdot (a - 6)_{3}^{12} \cdot (a + 7)_{3}^{12} \end{array} $ & $ \begin{array}{l} (2 a - 11)_{13}^{10} \cdot (2)^{20} \end{array}$ & \parbox{\curvewidth}{$ \left(-1181 a + 7035\right) x^{6} +\\ \left(18395 a - 104353\right) x^{5} + \left(-116071 a + 664673\right) x^{4} + \\ \left(386042 a - 2282384\right) x^{3} + \left(-742970 a + 4253365\right) x^{2} + \\ \left(784564 a - 4063679\right) x - 253294 a + 2224205 $}    \\  \hline
$ \left[17, 5, 2\right] $ & $ \left[8, 10, 17\right] $ & $ \frac{1}{2} \alpha^{2} + \frac{5}{2} $ & $ \begin{array}{l} 1 \end{array} $ & $ \begin{array}{l} (3 a + 1)_{17}^{10} \cdot (a)_{2}^{30} \end{array}$ & \parbox{\curvewidth}{$ \left(-3 a + 4\right) x^{5} - x^{4} + \left(6 a - 2\right) x^{3} + \left(9 a - 5\right) x^{2} + \\ \left(-3 a + 8\right) x - 3 a + 6 $}     \\  \hline
$ \left[17, 15, 52\right] $ & $ \left[13, 9, 17\right] $ & $ \alpha^{2} + 4 $ & $ \begin{array}{l} (a)_{3}^{12} \end{array} $ & $ \begin{array}{l} (a - 4)_{17}^{10} \cdot (2)^{20} \end{array}$ & \parbox{\curvewidth}{$ -x^{6} - 2 a x^{5} + \left(3 a - 3\right) x^{4} +  \left(8 a + 4\right) x^{3} + \left(-19 a + 39\right) x^{2} + \\ \left(16 a - 30\right) x + 3 a - 36 $}    \\  \hline
\multirow{2}{*}{$ \left[17, 25, 50\right] $} & \multirow{2}{*}{$ \left[8, 50, 425\right] $} & \multirow{2}{*}{$ \frac{1}{10} \alpha^{2} + \frac{5}{2} $} & $ \begin{array}{l} (a)_{2}^{24} \cdot (2 a + 1)_{7}^{12} \end{array} $ & $ \begin{array}{l} (3 a + 1)_{17}^{10} \cdot (5)^{10} \end{array}$ & \parbox{\curvewidth}{$ \left(6 a - 2\right) x^{6} + \left(-50 a - 64\right) x^{5} + \left(285 a + 485\right) x^{4} + \\ \left(-485 a - 435\right) x^{3} + \left(-70 a + 90\right) x^{2} + \\ \left(244 a + 92\right) x + 70 a - 166 $}     \\  \cline{4-7}
& & & $ \begin{array}{l} (a)_{2}^{36} \cdot (a + 7)_{47}^{12}\\ \cdot (2 a + 1)_{7}^{12} \end{array} $ & $ \begin{array}{l} (3 a + 1)_{17}^{10} \cdot (5)^{10} \end{array}$ & \parbox{\curvewidth}{$ \left(315 a + 422\right) x^{6} + \left(1212 a + 1757\right) x^{5} + \left(-2605 a - 3240\right) x^{4} + \\ \left(-50 a - 625\right) x^{3} + \left(1730 a - 570\right) x^{2} + \\ \left(864 a - 212\right) x + 72 a + 456 $}     \\  \hline
$ \left[17, 46, 257\right] $ & $ \left[257, 23, 68\right] $ & $ \alpha^{2} + 11 $ & $ \begin{array}{l} 
\left(11, a + 5\right)^{12} \cdot \left(13, a + 10\right)^{12}\\ \cdot \left(2, a\right)^{12} \cdot \left(2, a + 1\right)^{24}\\ \cdot \left(59, a + 14\right)^{12}
\end{array} $ & $ \begin{array}{l}
\left(17, a + 6\right)^{10} \cdot \left(2, a + 1\right)^{20}
\end{array}$ & \parbox{\curvewidth}{$ \left(-22 a - 1802\right) x^{6} +\\ \left(3596 a + 11488\right) x^{5} + \left(-30700 a - 354072\right) x^{4} + \\ \left(243927 a + 1843299\right) x^{3} + \left(-616892 a - 5576996\right) x^{2} + \\ \left(647768 a + 5283496\right) x - 198146 a - 1755298 $}     \\  \hline
$ \left[17, 47, 548\right] $ & $ \left[137, 35, 272\right] $ & $ \alpha^{2} + 17 $ & $ \begin{array}{l} (14 a - 75)_{11}^{12} \cdot (4 a + 25)_{19}^{12} \\ \cdot (3 a - 16)_{2}^{12} \cdot (3 a + 19)_{2}^{24} \end{array} $ & $ \begin{array}{l} (8 a + 51)_{17}^{10} \end{array}$ & \parbox{\curvewidth}{$ \left(285 a + 1620\right) x^{6} +\\ \left(-2683 a - 19110\right) x^{5} + \left(13341 a + 76698\right) x^{4} + \\ \left(-28642 a - 195577\right) x^{3} + \left(40284 a + 245904\right) x^{2} + \\ \left(-27600 a - 177408\right) x + 8154 a + 51670 $}    \\  \hline
\multirow{2}{*}{$ \left[29, 7, 5\right] $} & \multirow{2}{*}{$ \left[5, 11, 29\right] $} & \multirow{2}{*}{$ \alpha^{2} + 5 $} & $ \begin{array}{l} (2)^{12} \cdot (2 a + 1)_{5}^{12} \end{array} $ & $ \begin{array}{l} (a - 5)_{29}^{10} \end{array}$ & \parbox{\curvewidth}{$ \left(-4 a - 5\right) x^{6} + \left(11 a + 37\right) x^{5} + \left(-65 a - 62\right) x^{4} + \\ \left(111 a + 104\right) x^{3} + \left(-28 a - 189\right) x^{2} + \\ \left(-28 a + 157\right) x - 19 a - 76 $}     \\  \cline{4-7}
& & & $ \begin{array}{l} (2)^{12} \cdot (5 a + 3)_{31}^{12} \cdot (2 a + 1)_{5}^{12} \end{array} $ & $ \begin{array}{l} (a - 5)_{29}^{10} \end{array}$ & \parbox{\curvewidth}{$ \left(18 a + 42\right) x^{6} + \left(62 a + 194\right) x^{5} + \left(-209 a + 31\right) x^{4} + \\ \left(-648 a - 471\right) x^{3} + \left(116 a + 338\right) x^{2} + \\ \left(244 a + 259\right) x - 65 a - 159 $}     \\  \hline
$ \left[29, 9, 13\right] $ & $ \left[13, 18, 29\right] $ & $ \frac{1}{4} \alpha^{2} + \frac{7}{4} $ & $ \begin{array}{l} (a)_{3}^{12} \end{array} $ & $ \begin{array}{l} (2)^{20} \cdot (3 a + 2)_{29}^{10} \end{array}$ & \parbox{\curvewidth}{$ \left(-25 a + 56\right) x^{6} + \left(172 a - 39\right) x^{5} + \left(-39 a + 561\right) x^{4} + \\ \left(312 a + 234\right) x^{3} + \left(73 a + 354\right) x^{2} + \\ \left(76 a + 141\right) x + 15 a + 37 $}    \\  \hline
\multirow{2}{*}{$ \left[29, 21, 45\right] $} & \multirow{2}{*}{$ \left[5, 33, 261\right] $} & \multirow{2}{*}{$ \frac{1}{3} \alpha^{2} + 5 $} & $ \begin{array}{l} (4 a + 1)_{19}^{12} \cdot (3)^{6} \end{array} $ & $ \begin{array}{l} (2)^{20} \cdot (a - 5)_{29}^{10} \end{array}$ & \parbox{\curvewidth}{$ \left(-a + 20\right) x^{6} + \left(-87 a - 18\right) x^{5} + \left(-48 a + 198\right) x^{4} + \\ \left(-8 a - 296\right) x^{3} + \left(384 a + 360\right) x^{2} + \\ \left(-384 a - 480\right) x + 144 a + 216 $}     \\  \cline{4-7}
& & & $ \begin{array}{l} (3)^{6} \end{array} $ & $ \begin{array}{l} (2)^{20} \cdot (a - 5)_{29}^{10} \end{array}$ & \parbox{\curvewidth}{$ \left(-102 a - 165\right) x^{5} + \left(45 a + 72\right) x^{4} +  \left(-174 a - 262\right) x^{3} +\\ \left(36 a - 66\right) x^{2} +  \left(69 a - 144\right) x + 5 a - 107 $}     \\  \hline
$ \left[29, 26, 53\right] $ & $ \left[53, 13, 29\right] $ & $ \alpha^{2} + 6 $ & $ \begin{array}{l} (a - 1)_{11}^{12} \cdot (a + 1)_{13}^{12} \\ \cdot (a + 6)_{17}^{12} \end{array} $ & $ \begin{array}{l} (2)^{20} \cdot (a - 6)_{29}^{10} \end{array}$ & \parbox{\curvewidth}{$ \left(-790 a + 1564\right) x^{6} +\\ \left(241 a - 12431\right) x^{5} + \left(-15139 a - 14345\right) x^{4} + \\ \left(-2950 a - 165614\right) x^{3} + \left(-51588 a - 116086\right) x^{2} + \\ \left(-58139 a - 53507\right) x + 12653 a - 123381 $}    \\  \hline
$ \left[41, 11, 20\right] $ & $ \left[5, 13, 41\right] $ & $ \alpha^{2} + 6 $ & $ \begin{array}{l} 1 \end{array} $ & $ \begin{array}{l} (2)^{20} \cdot (a - 6)_{41}^{10} \end{array}$ & \parbox{\curvewidth}{$ \left(a + 4\right) x^{6} + \left(6 a - 2\right) x^{5} + 17 x^{4} + \left(-12 a - 16\right) x^{3} + \left(24 a - 5\right) x^{2} + \\ \left(-54 a - 16\right) x + 33 a + 9 $}    \\  \hline
$ \left[53, 13, 29\right] $ & $ \left[29, 26, 53\right] $ & $ \frac{1}{4} \alpha^{2} + \frac{11}{4} $ & $ \begin{array}{l} (a + 6)_{23}^{12} \cdot (a - 1)_{5}^{12} \\ \cdot (a)_{7}^{12} \end{array} $ & $ \begin{array}{l} (2)^{20} \cdot (3 a + 5)_{53}^{10} \end{array}$ & \parbox{\curvewidth}{$ \left(-31 a + 70\right) x^{6} + \left(151 a - 322\right) x^{5} + \left(-405 a + 658\right) x^{4} + \\ \left(238 a - 846\right) x^{3} + \left(3288 a + 2437\right) x^{2} + \\ \left(-3262 a + 12157\right) x - 27420 a - 58255 $}    \\  \hline
$ \left[61, 9, 5\right] $ & $ \left[5, 17, 61\right] $ & $ \frac{1}{3} \alpha^{2} + \frac{7}{3} $ & $ \begin{array}{l} 1 \end{array} $ & $ \begin{array}{l} (2)^{20} \cdot (7 a + 4)_{61}^{10} \end{array}$ & \parbox{\curvewidth}{$ \left(a + 2\right) x^{6} + \left(-2 a - 15\right) x^{5} + \left(36 a - 4\right) x^{4} +  \left(72 a + 24\right) x^{3} + \\ \left(8 a - 24\right) x^{2} + \left(-48 a - 80\right) x - 24 a - 40 $}    \\  \hline
$ \left[73, 9, 2\right] $ & $ \left[8, 18, 73\right] $ & $ \frac{1}{2} \alpha^{2} + \frac{9}{2} $ & $ \begin{array}{l} (a)_{2}^{24} \cdot (2 a - 1)_{7}^{12} \end{array} $ & $ \begin{array}{l} (2 a - 9)_{73}^{10} \end{array}$ & \parbox{\curvewidth}{$ \left(-12 a - 6\right) x^{6} + \left(8 a + 82\right) x^{5} + \left(-51 a + 92\right) x^{4} + \\ \left(-126 a - 1\right) x^{3} + \left(-36 a + 35\right) x^{2} + \\ \left(32 a + 50\right) x + 10 a + 8 $}     \\  \hline
$ \left[73, 47, 388\right] $ & $ \left[97, 94, 657\right] $ & $ \frac{1}{8} \alpha^{2} + \frac{43}{8} $ & $ \begin{array}{l} (20 a + 109)_{101}^{12} \cdot (7 a + 38)_{2}^{24} \\ \cdot (7 a - 31)_{2}^{12} \cdot (2 a - 9)_{3}^{12} \\ \cdot (2 a + 11)_{3}^{12} \cdot (30 a + 163)_{79}^{12} \end{array} $ & $ \begin{array}{l} (22 a + 119)_{73}^{10} \end{array}$ & \parbox{\curvewidth}{$ \left(23 a - 43\right) x^{6} + \left(-149 a - 1221\right) x^{5} + \left(8675 a + 44883\right) x^{4} + \\ \left(-128038 a - 698079\right) x^{3} + \left(928849 a + 5037588\right) x^{2} + \\ \left(123515 a + 671208\right) x + 4023 a + 21640 $}    \\  \hline
$ \left[89, 11, 8\right] $ & $ \left[8, 22, 89\right] $ & $ \frac{1}{4} \alpha^{2} + \frac{11}{4} $ & $ \begin{array}{l} (a)_{2}^{24} \end{array} $ & $ \begin{array}{l} (7 a + 3)_{89}^{10} \end{array}$ & \parbox{\curvewidth}{$ -x^{5} + \left(-4 a + 2\right) x^{4} +  21 x^{3} + \left(-16 a + 64\right) x^{2} \\- 160 x + 142 a - 190 $}     \\  \hline
$ \left[97, 94, 657\right] $ & $ \left[73, 47, 388\right] $ & $ \frac{1}{3} \alpha^{2} + \frac{22}{3} $ & $ \begin{array}{l} (a - 4)_{2}^{12} \cdot (a + 5)_{2}^{12}\\ \cdot (14 a - 53)_{23}^{12} \cdot (4 a - 15)_{3}^{12}\\ \cdot (4 a + 19)_{3}^{12} \cdot (30 a + 143)_{41}^{12}\\ \cdot (10 a + 47)_{61}^{12} \end{array} $ & $ \begin{array}{l} (24 a + 115)_{97}^{10} \end{array}$ & \parbox{\curvewidth}{$ \left(-128252 a - 611298\right) x^{6} +\\ \left(-984572 a - 4709700\right) x^{5} + \left(-3071730 a - 15394554\right) x^{4} + \\ \left(-6889006 a - 20077475\right) x^{3} + \left(-39650571 a + 105355350\right) x^{2} + \\ \left(174191751 a - 679664106\right) x + 256866525 a - 973717416 $}    \\  \hline
\multirow{2}{*}{$ \left[101, 33, 45\right] $} & \multirow{2}{*}{$ \left[5, 66, 909\right] $} & \multirow{2}{*}{$ \frac{1}{12} \alpha^{2} + \frac{9}{4} $} & $ \begin{array}{l} (3)^{6} \cdot (2 a + 1)_{5}^{12} \cdot (7 a + 3)_{61}^{12} \end{array} $ & $ \begin{array}{l} (9 a + 5)_{101}^{10} \cdot (2)^{20} \end{array}$ & \parbox{\curvewidth}{$ \left(-216 a + 464\right) x^{6} + \left(-2304 a - 48\right) x^{5} + \left(-3984 a - 960\right) x^{4} + \\ \left(-864 a + 3088\right) x^{3} + \left(-720 a + 1422\right) x^{2} + \\ \left(-4047 a - 5322\right) x - 818 a - 2423 $}     \\  \cline{4-7}
& & & $ \begin{array}{l} (4 a + 3)_{19}^{12} \cdot (4 a + 1)_{19}^{12}\\ \cdot (3)^{6} \cdot (5 a + 3)_{31}^{12}\\ \cdot (2 a + 1)_{5}^{12} \end{array} $ & $ \begin{array}{l} (9 a + 5)_{101}^{10} \cdot (2)^{20} \end{array}$ & \parbox{\curvewidth}{$ \left(-5229 a + 4019\right) x^{6} +\\ \left(-6132 a - 6909\right) x^{5} + \left(44637 a - 2364\right) x^{4} + \\ \left(53094 a + 58660\right) x^{3} + \left(-39159 a + 19266\right) x^{2} + \\ \left(-30363 a - 55761\right) x - 16848 a - 16911 $}     \\  \hline
$ \left[109, 17, 45\right] $ & $ \left[5, 21, 109\right] $ & $ \alpha^{2} + 10 $ & $ \begin{array}{l} (2 a + 1)_{5}^{12} \end{array} $ & $ \begin{array}{l} (a - 10)_{109}^{10} \cdot (2)^{20} \end{array}$ & \parbox{\curvewidth}{$ \left(-8 a - 8\right) x^{6} - 16 x^{5} + \left(8 a + 72\right) x^{4} +  \left(152 a + 184\right) x^{3} + \\ \left(6 a + 84\right) x^{2} +  \left(-255 a - 339\right) x - 319 a - 524 $}    \\  \hline
\multirow{2}{*}{$ \left[113, 33, 18\right] $} & \multirow{2}{*}{$ \left[8, 66, 1017\right] $} & \multirow{2}{*}{$ \frac{1}{6} \alpha^{2} + \frac{11}{2} $} & $ \begin{array}{l} (3 a + 11)_{103}^{12} \cdot (a)_{2}^{24} \\ \cdot (3)^{6} \cdot (4 a - 1)_{31}^{12} \\ \cdot (2 a - 1)_{7}^{12} \cdot (2 a + 1)_{7}^{12} \end{array} $ & $ \begin{array}{l} (2 a - 11)_{113}^{10} \end{array}$ & \parbox{\curvewidth}{$ \left(122 a + 800\right) x^{6} + \left(-1509 a - 909\right) x^{5} + \left(36762 a - 85470\right) x^{4} + \\ \left(-116871 a + 265713\right) x^{3} + \left(-467682 a + 704460\right) x^{2} + \\ \left(-480528 a + 365352\right) x - 7616 a + 226442 $}     \\  \cline{4-7}
& & & $ \begin{array}{l} (a)_{2}^{24} \cdot (3)^{6}\\ \cdot (4 a + 1)_{31}^{12} \cdot (2 a + 1)_{7}^{12} \end{array} $ & $ \begin{array}{l} (2 a - 11)_{113}^{10} \end{array}$ & \parbox{\curvewidth}{$ \left(-418 a - 190\right) x^{6} + \left(1476 a - 660\right) x^{5} + \left(1146 a + 6810\right) x^{4} + \\ \left(2145 a + 2175\right) x^{3} + \left(-1437 a - 3489\right) x^{2} + \\ \left(-42 a - 2736\right) x + 830 a + 394 $}     \\  \hline
$ \left[137, 35, 272\right] $ & $ \left[17, 47, 548\right] $ & $ \alpha^{2} + 23 $ & $ \begin{array}{l} (2 a - 5)_{19}^{12} \cdot (a + 2)_{2}^{12}\\ \cdot (a - 1)_{2}^{12} \end{array} $ & $ \begin{array}{l} (6 a - 1)_{137}^{10} \end{array}$ & \parbox{\curvewidth}{$ \left(4 a + 6\right) x^{6} + \left(8 a + 36\right) x^{5} + \left(-4 a + 42\right) x^{4} +  \left(586 a + 1289\right) x^{3} + \\ \left(1066 a + 2808\right) x^{2} +  4 a x + 25596 a + 65566 $}    \\  \hline
$ \left[149, 13, 5\right] $ & $ \left[5, 26, 149\right] $ & $ \frac{1}{4} \alpha^{2} + \frac{11}{4} $ & $ \begin{array}{l} (3 a + 1)_{11}^{12} \end{array} $ & $ \begin{array}{l} (11 a + 7)_{149}^{10} \cdot (2)^{20} \end{array}$ & \parbox{\curvewidth}{$ 8 x^{6} + 96 x^{5} + \left(-24 a + 168\right) x^{4} + \left(-576 a - 808\right) x^{3} +\\ \left(66 a - 132\right) x^{2} +  \left(292 a + 47\right) x + 86 a - 87 $}    \\  \hline
$ \left[157, 25, 117\right] $ & $ \left[13, 41, 157\right] $ & $ \frac{1}{9} \alpha^{2} + \frac{16}{9} $ & $ \begin{array}{l} (a - 4)_{17}^{12} \cdot (3 a - 1)_{23}^{12}\\ \cdot (a)_{3}^{24} \cdot (a + 1)_{3}^{12} \end{array} $ & $ \begin{array}{l} (7 a + 5)_{157}^{10} \cdot (2)^{20} \end{array}$ & \parbox{\curvewidth}{$ \left(-3328 a - 7633\right) x^{6} + \\ \left(-17510 a - 39323\right) x^{5} + \left(-32518 a - 68044\right) x^{4} + \\ \left(-17960 a - 66720\right) x^{3} + \left(256 a - 51704\right) x^{2} + \\ \left(5184 a - 22864\right) x + 1432 a - 5264 $}    \\  \hline
$ \left[181, 41, 13\right] $ & $ \left[13, 29, 181\right] $ & $ \frac{1}{3} \alpha^{2} + \frac{13}{3} $ & $ \begin{array}{l} (a + 5)_{17}^{12} \cdot (3 a + 2)_{29}^{12}\\ \cdot (a)_{3}^{24} \cdot (a + 1)_{3}^{12} \end{array} $ & $ \begin{array}{l} (3 a - 13)_{181}^{10} \cdot (2)^{20} \end{array}$ & \parbox{\curvewidth}{$ \left(330 a + 1417\right) x^{6} + \left(11102 a + 1701\right) x^{5} + \left(1396 a + 59742\right) x^{4} + \\ \left(24016 a + 92792\right) x^{3} + \left(74408 a + 38064\right) x^{2} + \\ \left(35248 a + 26160\right) x - 5784 a + 21888 $}    \\  \hline
$ \left[233, 19, 32\right] $ & $ \left[8, 38, 233\right] $ & $ \frac{1}{8} \alpha^{2} + \frac{19}{8} $ & $ \begin{array}{l} (a)_{2}^{24} \cdot (a - 5)_{23}^{12}\\ \cdot (a + 5)_{23}^{12} \cdot (2 a + 1)_{7}^{12} \end{array} $ & $ \begin{array}{l} (11 a + 3)_{233}^{10} \end{array}$ & \parbox{\curvewidth}{$ \left(2348 a - 3554\right) x^{6} + \left(11828 a - 12348\right) x^{5} + \left(4498 a - 23598\right) x^{4} + \\ \left(12704 a + 9133\right) x^{3} + \left(-3151 a - 14433\right) x^{2} + \\ \left(5344 a - 1974\right) x + 18 a - 604 $}     \\  \hline
$ \left[257, 23, 68\right] $ & $ \left[17, 46, 257\right] $ & $ \frac{1}{8} \alpha^{2} + \frac{19}{8} $ & $ \begin{array}{l} (2 a + 3)_{13}^{12} \cdot (a + 2)_{2}^{12}\\ \cdot (a - 1)_{2}^{24} \cdot (4 a - 3)_{43}^{12} \\ \cdot (2 a + 9)_{47}^{12} \cdot (4 a + 13)_{53}^{12} \end{array} $ & $ \begin{array}{l} (8 a - 19)_{257}^{10} \end{array}$ & \parbox{\curvewidth}{$ \left(-2809 a - 7326\right) x^{6} +\\ \left(5069 a + 3572\right) x^{5} + \left(52427 a - 51416\right) x^{4} + \\ \left(249518 a + 105951\right) x^{3} + \left(-311115 a - 180355\right) x^{2} + \\ \left(156533 a - 20215\right) x - 34657 a + 19003 $}     \\  \hline
$ \left[269, 17, 5\right] $ & $ \left[5, 34, 269\right] $ & $ \frac{1}{4} \alpha^{2} + \frac{15}{4} $ & $ \begin{array}{l} (3 a + 1)_{11}^{12} \cdot (2 a + 1)_{5}^{12} \end{array} $ & $ \begin{array}{l} (2)^{20} \cdot (15 a + 11)_{269}^{10} \end{array}$ & \parbox{\curvewidth}{$ \left(-168 a - 272\right) x^{6} + \left(960 a + 1696\right) x^{5} + \left(472 a - 1008\right) x^{4} + \\
 \left(-4448 a - 1552\right) x^{3} + \left(358 a + 904\right) x^{2} + \\ \left(945 a + 1690\right) x $}    \\  \hline
$ \left[281, 17, 2\right] $ & $ \left[8, 34, 281\right] $ & $ \frac{1}{2} \alpha^{2} + \frac{17}{2} $ & $ \begin{array}{l} (a)_{2}^{36} \cdot (4 a + 1)_{31}^{12} \\ \cdot (2 a - 1)_{7}^{12} \cdot (2 a + 1)_{7}^{12} \end{array} $ & $ \begin{array}{l} (2 a - 17)_{281}^{10} \end{array}$ & \parbox{\curvewidth}{$ \left(-835 a + 1960\right) x^{6} + \left(1343 a + 7589\right) x^{5} + \left(19630 a + 6428\right) x^{4} + \\ \left(26923 a + 13601\right) x^{3} + \left(-6743 a + 44228\right) x^{2} + \\ \left(-5762 a + 18262\right) x + 17138 a - 23184 $}     \\  \hline
$ \left[389, 37, 245\right] $ & $ \left[5, 41, 389\right] $ & $ \frac{1}{5} \alpha^{2} + \frac{18}{5} $ & $ \begin{array}{l} (3 a + 1)_{11}^{12} \cdot (3 a + 2)_{11}^{12} \\ \cdot (4 a + 3)_{19}^{12} \cdot (4 a + 1)_{19}^{12}\\ \cdot (a + 6)_{29}^{12} \cdot (2 a + 1)_{5}^{12} \end{array} $ & $ \begin{array}{l} (2)^{20} \cdot (18 a + 13)_{389}^{10} \end{array}$ & \parbox{\curvewidth}{$ \left(-22952 a - 6848\right) x^{6} +\\ \left(162272 a - 61136\right) x^{5} + \left(296568 a +208208\right) x^{4} + \\ \left(-212600 a - 959344\right) x^{3} + \left(89874 a + 1610270\right) x^{2} + \\ \left(-428348 a - 1023457\right) x + 315516 a + 343397 $}    \\  \hline

\end{longtable}

\label{table2bend}
\end{landscape}
\addtolength{\voffset}{-3cm}
\addtolength{\hoffset}{1.1in}
\addtolength{\textwidth}{-160pt}

\begin{table}
\caption*{Table 2c}
\centering
\begin{tabular}{|l|l|l|l|c|} \hline
DAB & DAB reflex & $a$ & obstructions & curves \\ \hline

 $ \left[8, 14, 41\right] $ & $ \left[41, 7, 2\right] $ & $ \alpha^{2} + 3 $ & $ (a + 4)_{2} , (a - 3)_{2} $ & $ 2 $     \\ \hline
 $ \left[8, 26, 137\right] $ & $ \left[137, 13, 8\right] $ & $ \alpha^{2} + 6 $ & $ (3 a - 16)_{2} , (3 a + 19)_{2} $ & $ 2 $      \\ \hline
 $ \left[8, 30, 153\right] $ & $ \left[17, 15, 18\right] $ & $ \frac{1}{3} \alpha^{2} + 2 $ & $ (a + 2)_{2} , (a - 1)_{2} $ & $ 4 $      \\ \hline
 $ \left[12, 8, 13\right] $ & $ \left[13, 10, 12\right] $ & $ \frac{1}{2} \alpha^{2} + 2 $ & $ (a + 1)_{3} , (2) $ & $ 2 $      \\ \hline
 $ \left[12, 10, 13\right] $ & $ \left[13, 5, 3\right] $ & $ \alpha^{2} + 2 $ & $ (a + 1)_{3} , (2) $ & $ 2 $      \\ \hline
 $ \left[12, 14, 37\right] $ & $ \left[37, 7, 3\right] $ & $ \alpha^{2} + 3 $ & $ (a + 3)_{3} , (2) $ & $ 2 $      \\ \hline
 $ \left[12, 26, 61\right] $ & $ \left[61, 13, 27\right] $ & $ \alpha^{2} + 6 $ & $ (a - 3)_{3} , (2) $ & $ 2 $      \\ \hline
 $ \left[12, 26, 157\right] $ & $ \left[157, 13, 3\right] $ & $ \alpha^{2} + 6 $ & $ (a - 6)_{3} , (2) $ & $ 2 $      \\ \hline
 $ \left[12, 50, 325\right] $ & $ \left[13, 25, 75\right] $ & $ \frac{1}{5} \alpha^{2} + 2 $ & $ (a + 1)_{3} , (2) $ & $ 4 $      \\ \hline
 $ \left[44, 8, 5\right] $ & $ \left[5, 14, 44\right] $ & $ \frac{1}{2} \alpha^{2} + 3 $ & $ (2) , (3 a + 2)_{11} $ & $ 2 $      \\ \hline
 $ \left[44, 14, 5\right] $ & $ \left[5, 7, 11\right] $ & $ \alpha^{2} + 3 $ & $ (2) , (3 a + 2)_{11} $ & $ 2 $      \\ \hline
 $ \left[44, 42, 45\right] $ & $ \left[5, 21, 99\right] $ & $ \frac{1}{3} \alpha^{2} + 3 $ & $ (2) , (3 a + 2)_{11} $ & $ 4 $      \\ \hline
 $ \left[76, 18, 5\right] $ & $ \left[5, 9, 19\right] $ & $ \alpha^{2} + 4 $ & $ (2) , (4 a + 3)_{19} $ & $ 2 $      \\ \hline
 $ \left[172, 34, 117\right] $ & $ \left[13, 17, 43\right] $ & $ \frac{1}{3} \alpha^{2} + \frac{7}{3} $ & $ (2) , (4 a + 5)_{43} $ & $ 2 $      \\ \hline
 $ \left[236, 32, 20\right] $ & $ \left[5, 16, 59\right] $ & $ \frac{1}{2} \alpha^{2} + \frac{7}{2} $ & $ (2) , (7 a + 5)_{59} $ & $ 2 $      \\ \hline
\end{tabular}
\end{table}

\section{Application}
\label{sec:application}

Obviously, we hope that our list is useful for experimenting
with complex multiplication and hyperelliptic curves.
Additionally, this final section gives a cryptographic
application: the small coefficients of the curves
in our table allow for faster communication and arithmetic.

Cryptographic hyperelliptic curves are constructed
as follows using the theory of complex multiplication
(for details, see~\cite{hehcc18}).
\begin{enumerate}
\item
Compute the Igusa invariants $I_n(\widetilde{C})$ of
a genus-two curve $\widetilde{C}$ with CM
by an order~$\mathcal{O}_K$
over a number field~$L$.
\item Reduce
these invariants modulo a prime~$\mathfrak{p}$ of~$L$,
which yields elements of the residue field~$k=\mathcal{O}_L/\mathfrak{p}$.
\item
Construct a curve~$C$ over the finite field~$k$
with these invariants, using Mestre's algorithm.
\end{enumerate}
Then there is a relation between the CM-type
$(K,\Phi)$ of $\widetilde{C}$ and the number
of
$k$-points in the Jacobian groups of $C$
and its quadratic twist $C'$.
So with a good choice of $\Phi$ and $\mathfrak{p}$,
we can construct curves $C$
for which $J_C(k)$ has a prescribed prime order,
or other interesting cryptographic properties.

In the end, the coefficients of the curve are random-looking
elements of~$k$, so if $k$ has $q$ elements, these coefficients
take up about $\log_{10}(q)$ digits each,
where $q$ is a cryptographically large prime power.

Now if the CM-field~$K$
is one of the fields in our table,
we can do better: we can take $\widetilde{C}$
from our table, and let $C$
be $(\widetilde{C}\ \mathrm{mod}\ \mathfrak{p})$.
This curve then has coefficients of a simple
and elegant shape.
This saves bandwidth when communicating this curve.
It also saves ``carries''
in multiplication operations involving curve coefficients,
making them potentially much more efficient.

For example, in~\cite[Section~8, Example of Algorithm~3]{hmns}
a curve $C$ is constructed following the recipe 1.,2.,3.~with~$\Phi$ a certain CM-type of $K=[5,13,41]$
This curve is defined over a finite field
$k=\FF_{p^2}$, where
\small
\begin{align*}p
=&
 142003856595807482747635387048977088071520136032341569\\
&014612056864049709760143646636956724980664377491196079\\
&730519617723521029855649462172148699393958968638652107\\
&696147277436345811056227385195781997362304851932650270\\
&514293705125991379
\end{align*}
\normalsize
and $J_C(k)$ has a cryptographic
subgroup of order $r=2^{192}+18513$.
The curve $C$ is given by $C:y^2=\sum_{n=0}^6 {a_n} x^n$,
and simple transformations make $a_6$ small (either~$1$
or a small non-square in~$k$) and ensure $a_5=0$.
Then there are five coefficients $a_0,\ldots,a_4\in\FF_{p^2}$,
each taking up twice as much space as the number~$p$
written above,
hence more than $2000$ digits in total.

Now let us look up $K=[5,13,41]$ in the table.
Let $a$ be a root of $X^2+1-10$ over~$\QQ$.
We find that 
up to twist and up to conjugation of $K_0^r=\QQ(a)/\QQ$,
we have $\widetilde{C}:y^2 = f(x)$, where
$$f(x) = \left(-a + 3\right) x^{6} + \left(4 a - 8\right) x^{5} + 10 x^{4} + \left(-a + 20\right) x^{3} + \left(4 a + 5\right) x^{2} + \left(a + 4\right) x + 1.$$
Consequently, if we write by abuse of notation
$a$ also for a root of $X^2+X-10$ generating a quadratic
extension $k=\FF_{p^2}/\FF_p$, then the curve~$C$ is given by the same equation. Again up to twist and conjugation of $k/\FF_p$.

Conjugation does not affect the number of points of~$C(k)$,
and as $(a-2)$ is a non-square in~$k^*$, we find that the only
non-isomorphic twist of~$C$ is 
given by $y^2 = (a-2)f(x)$, which also has
very simple coefficients.
So one of these curves could take the place of the curve in~\cite{hmns}.

\begin{remark}
For completeness, we determine which twist of~$C$
gives a subgroup of order~$r$ in~$J_C(k)$.
Let $\pi\in K$ be the Frobenius endomorphism of~$C$.
Then $(\pi)=\mathfrak{p}_1^2\mathfrak{p}_2$, where
$p\mathcal{O}_K = \mathfrak{p}_1\overline{\mathfrak{p}_1}\mathfrak{p}_2$ by \cite[Lemma 21]{hmns}.
This fixes $\pi$ up to complex conjugation and roots of unity,
hence gives two candidates $N(\pi-1)$ and $N(-\pi-1)$
for the order of $J_C(k)$.
We compute these candidates and find that one of them,
let us call it~$n_1$, is divisible by $r$ and the other, $n_2$,
is not.
Now let $D=2 (0,1) - \infty$, that is,
$D$ is the divisor given by twice $P=(0,1)$ minus
both points at infinity.
We use Magma to
check $n_2 [D]\not=0\in J_C(k)$, which proves
$\#J_C(k)=n_1$, so~$C$ is itself the correct twist
(and indeed we easily verify $n_1[D]=0$).
We also check $(n_1/r) [D]\not=0\in J_C(k)$, which proves
that $(n_1/r)[D]$ generates the group of order $r$ in~$J_C(k)$.
\end{remark}

The following theorem gives our CM construction
as a canned result.
\begin{theorem}\label{thm:cannedcm}
Let $K$, $K^{\mathrm{r}}$, $f$, and~$\Delta(C)$
be as in an entry of Table 1a, 1b, or~2b
other than $\mathrm{DAB}=[5,5,5]$.
Let $\mathfrak{p} \nmid \Delta(C)$
be a prime of ${K_0^{\mathrm{r}}}$
that is not inert
in $K^{\mathrm{r}}/K_0^{\mathrm{r}}$
and let $k_{\mathfrak{p}}$ be its residue field.
Let $\overline{f}=(f\mod \mathfrak{p})$
and let $b\in k_{\mathfrak{p}}^*$ be a non-square.
Let $C_1$, $C_2$ be the curves
$y^2=\overline{f}$ and $y^2=b\overline{f}$
over $k_{\mathfrak{p}}$.

Let $\mathfrak{P}\mid\mathfrak{p}$ be a
prime of $K^{\mathrm{r}}$
and $\Phi^{\mathrm{r}}$ the CM-type of~$K^{\mathrm{r}}$
with reflex field~$K$ (uniquely determined up to complex conjugation).
Then the ideal $N_{\Phi^{\mathrm{r}}}(\mathfrak{P})\subset\mathcal{O}_K$ is principal, and generated
by an element~$\pi$ such that $\pi\overline{\pi}\in\QQ$.

Moreover, the endomorphism rings of $J(C_i)$
over $k_{\mathfrak{p}}$
contain subrings isomorphic to $\mathcal{O}_K$
and the isomorphisms can be chosen in such 
a way that $\{\mathrm{Frob}_{C_i, N(\mathfrak{p})}\} = \{\pm \pi\}$.
In particular, we have
$\{\#J(C_i)(k_{\mathfrak{p}})\}=\{N_{K/\QQ}(\pm \pi-1)\}$.
\end{theorem}
The computation of~$\pi\in\mathcal{O}_K$
is straightforward using algebraic number theory.
Deciding which of the~$C_i$ has Frobenius~$\pi$
and which has Frobenius $-\pi$ can be done
by checking whether a random point
on the Jacobian
is annihilated by
$N_{K/\QQ}(\pm \pi -1)$.

Note that we have a surjective map
$\mathcal{O}_{K_0}=\ZZ_p[X]/(X^2+\epsilon X+(\epsilon-D^{\mathrm{r}})/4)
\rightarrow k_{\mathfrak{p}}$,
and the coefficients of~$C_i$
are represented by small elements of the ring~$\mathcal{O}_{K_0}$,
hence operations in the group $J(C_i)$
can be performed with a smaller number of
carrying operations compared to when using
curves with random coefficients.
\begin{proof}[{Proof of Theorem~\ref{thm:cannedcm}}]
Our assumptions imply $k_\mathfrak{P}=k_{\mathfrak{p}}$,
and our $\Phi^{\mathrm{r}}$ is the
\emph{reflex} of~$\Phi$
as defined in~\cite{shimura}.
Our curves have Jacobians with
endomorphism ring $\mathcal{O}_K$
of type~$\Phi$ over $K^{\mathrm{r}}$
by Theorem~\ref{thm:mainlong}.
Moreover, they have good reduction
at~$\mathfrak{P}$ by~$\mathfrak{p}\nmid \Delta(C)$.
Therefore,
by the Shimura-Taniyama
formula
(\cite[Theorem~1(ii) in
Section~13.1]{shimura}
or \cite[Theorem~4.1.2]{lang-CM}),
we have $\mathrm{Frob}_{C_i, N(\mathfrak{p})}\mathcal{O}_K=N_{\Phi^{\mathrm{r}}}(\mathfrak{P})$.
This proves that the latter ideal has a generator
$\pi$ with $\pi\overline{\pi}\in\QQ$.
Such a generator is unique up to roots of unity,
of which $\mathcal{O}_K$ contains only~$\pm 1$.
Since~$b$ is a non-square, twisting by it
changes the root of unity, hence $\{\pm \pi\}$
occurs exactly for $\{\overline{f}, b\overline{f}\}$.
\end{proof}

  \bibliographystyle{plain}
\bibliography{bib}

\end{document}